\documentclass{amsart}
\usepackage{amsmath}
\usepackage{amsthm}
\usepackage{amssymb}
\usepackage{graphicx}
\usepackage{url}
\usepackage{comment}%???????????????\UTF{FFFD}\UTF{FFFD}??\UTF{FFFD}\UTF{FFFD}????\UTF{FFFD}\UTF{FFFD}??\UTF{FFFD}\UTF{FFFD}??????\UTF{FFFD}\UTF{FFFD}??\UTF{FFFD}\UTF{FFFD}????\UTF{FFFD}\UTF{FFFD}??\UTF{FFFD}\UTF{FFFD}????????\UTF{FFFD}\UTF{FFFD}??\UTF{FFFD}\UTF{FFFD}????\UTF{FFFD}\UTF{FFFD}??\UTF{FFFD}\UTF{FFFD}??????\UTF{FFFD}\UTF{FFFD}??\UTF{FFFD}\UTF{FFFD}????\UTF{FFFD}\UTF{FFFD}??\UTF{FFFD}\UTF{FFFD}??????????\UTF{FFFD}\UTF{FFFD}??\UTF{FFFD}\UTF{FFFD}????\UTF{FFFD}\UTF{FFFD}??\UTF{FFFD}\UTF{FFFD}??????\UTF{FFFD}\UTF{FFFD}??\UTF{FFFD}\UTF{FFFD}????\UTF{FFFD}\UTF{FFFD}??\UTF{FFFD}\UTF{FFFD}????????\UTF{FFFD}\UTF{FFFD}??\UTF{FFFD}\UTF{FFFD}????\UTF{FFFD}\UTF{FFFD}??\UTF{FFFD}\UTF{FFFD}??????\UTF{FFFD}\UTF{FFFD}??\UTF{FFFD}\UTF{FFFD}????\UTF{FFFD}\UTF{FFFD}??\UTF{FFFD}\UTF{FFFD}????????????\UTF{FFFD}\UTF{FFFD}??\UTF{FFFD}\UTF{FFFD}????\UTF{FFFD}\UTF{FFFD}??\UTF{FFFD}\UTF{FFFD}??????\UTF{FFFD}\UTF{FFFD}??\UTF{FFFD}\UTF{FFFD}????\UTF{FFFD}\UTF{FFFD}??\UTF{FFFD}\UTF{FFFD}????????\UTF{FFFD}\UTF{FFFD}??\UTF{FFFD}\UTF{FFFD}????\UTF{FFFD}\UTF{FFFD}??\UTF{FFFD}\UTF{FFFD}??????\UTF{FFFD}\UTF{FFFD}??\UTF{FFFD}\UTF{FFFD}????\UTF{FFFD}\UTF{FFFD}??\UTF{FFFD}\UTF{FFFD}??????????\UTF{FFFD}\UTF{FFFD}??\UTF{FFFD}\UTF{FFFD}????\UTF{FFFD}\UTF{FFFD}??\UTF{FFFD}\UTF{FFFD}??????\UTF{FFFD}\UTF{FFFD}??\UTF{FFFD}\UTF{FFFD}????\UTF{FFFD}\UTF{FFFD}??\UTF{FFFD}\UTF{FFFD}????????\UTF{FFFD}\UTF{FFFD}??\UTF{FFFD}\UTF{FFFD}????\UTF{FFFD}\UTF{FFFD}??\UTF{FFFD}\UTF{FFFD}??????\UTF{FFFD}\UTF{FFFD}??\UTF{FFFD}\UTF{FFFD}????\UTF{FFFD}\UTF{FFFD}??\UTF{FFFD}\UTF{FFFD}??????????????\UTF{FFFD}\UTF{FFFD}??\UTF{FFFD}\UTF{FFFD}????\UTF{FFFD}\UTF{FFFD}??\UTF{FFFD}\UTF{FFFD}??????\UTF{FFFD}\UTF{FFFD}??\UTF{FFFD}\UTF{FFFD}????\UTF{FFFD}\UTF{FFFD}??\UTF{FFFD}\UTF{FFFD}????????\UTF{FFFD}\UTF{FFFD}??\UTF{FFFD}\UTF{FFFD}????\UTF{FFFD}\UTF{FFFD}??\UTF{FFFD}\UTF{FFFD}??????\UTF{FFFD}\UTF{FFFD}??\UTF{FFFD}\UTF{FFFD}????\UTF{FFFD}\UTF{FFFD}??\UTF{FFFD}\UTF{FFFD}??????????\UTF{FFFD}\UTF{FFFD}??\UTF{FFFD}\UTF{FFFD}????\UTF{FFFD}\UTF{FFFD}??\UTF{FFFD}\UTF{FFFD}??????\UTF{FFFD}\UTF{FFFD}??\UTF{FFFD}\UTF{FFFD}????\UTF{FFFD}\UTF{FFFD}??\UTF{FFFD}\UTF{FFFD}????????\UTF{FFFD}\UTF{FFFD}??\UTF{FFFD}\UTF{FFFD}????\UTF{FFFD}\UTF{FFFD}??\UTF{FFFD}\UTF{FFFD}??????\UTF{FFFD}\UTF{FFFD}??\UTF{FFFD}\UTF{FFFD}????\UTF{FFFD}\UTF{FFFD}??\UTF{FFFD}\UTF{FFFD}????????????\UTF{FFFD}\UTF{FFFD}??\UTF{FFFD}\UTF{FFFD}????\UTF{FFFD}\UTF{FFFD}??\UTF{FFFD}\UTF{FFFD}??????\UTF{FFFD}\UTF{FFFD}??\UTF{FFFD}\UTF{FFFD}????\UTF{FFFD}\UTF{FFFD}??\UTF{FFFD}\UTF{FFFD}????????\UTF{FFFD}\UTF{FFFD}??\UTF{FFFD}\UTF{FFFD}????\UTF{FFFD}\UTF{FFFD}??\UTF{FFFD}\UTF{FFFD}??????\UTF{FFFD}\UTF{FFFD}??\UTF{FFFD}\UTF{FFFD}????\UTF{FFFD}\UTF{FFFD}??\UTF{FFFD}\UTF{FFFD}??????????\UTF{FFFD}\UTF{FFFD}??\UTF{FFFD}\UTF{FFFD}????\UTF{FFFD}\UTF{FFFD}??\UTF{FFFD}\UTF{FFFD}??????\UTF{FFFD}\UTF{FFFD}??\UTF{FFFD}\UTF{FFFD}????\UTF{FFFD}\UTF{FFFD}??\UTF{FFFD}\UTF{FFFD}????????\UTF{FFFD}\UTF{FFFD}??\UTF{FFFD}\UTF{FFFD}????\UTF{FFFD}\UTF{FFFD}??\UTF{FFFD}\UTF{FFFD}??????\UTF{FFFD}\UTF{FFFD}??\UTF{FFFD}\UTF{FFFD}????\UTF{FFFD}\UTF{FFFD}??\UTF{FFFD}\UTF{FFFD}?????????????????\UTF{FFFD}\UTF{FFFD}??\UTF{FFFD}\UTF{FFFD}????\UTF{FFFD}\UTF{FFFD}??\UTF{FFFD}\UTF{FFFD}??????\UTF{FFFD}\UTF{FFFD}??\UTF{FFFD}\UTF{FFFD}????\UTF{FFFD}\UTF{FFFD}??\UTF{FFFD}\UTF{FFFD}????????\UTF{FFFD}\UTF{FFFD}??\UTF{FFFD}\UTF{FFFD}????\UTF{FFFD}\UTF{FFFD}??\UTF{FFFD}\UTF{FFFD}??????\UTF{FFFD}\UTF{FFFD}??\UTF{FFFD}\UTF{FFFD}????\UTF{FFFD}\UTF{FFFD}??\UTF{FFFD}\UTF{FFFD}??????????\UTF{FFFD}\UTF{FFFD}??\UTF{FFFD}\UTF{FFFD}????\UTF{FFFD}\UTF{FFFD}??\UTF{FFFD}\UTF{FFFD}??????\UTF{FFFD}\UTF{FFFD}??\UTF{FFFD}\UTF{FFFD}????\UTF{FFFD}\UTF{FFFD}??\UTF{FFFD}\UTF{FFFD}????????\UTF{FFFD}\UTF{FFFD}??\UTF{FFFD}\UTF{FFFD}????\UTF{FFFD}\UTF{FFFD}??\UTF{FFFD}\UTF{FFFD}??????\UTF{FFFD}\UTF{FFFD}??\UTF{FFFD}\UTF{FFFD}????\UTF{FFFD}\UTF{FFFD}??\UTF{FFFD}\UTF{FFFD}????????????\UTF{FFFD}\UTF{FFFD}??\UTF{FFFD}\UTF{FFFD}????\UTF{FFFD}\UTF{FFFD}??\UTF{FFFD}\UTF{FFFD}??????\UTF{FFFD}\UTF{FFFD}??\UTF{FFFD}\UTF{FFFD}????\UTF{FFFD}\UTF{FFFD}??\UTF{FFFD}\UTF{FFFD}????????\UTF{FFFD}\UTF{FFFD}??\UTF{FFFD}\UTF{FFFD}????\UTF{FFFD}\UTF{FFFD}??\UTF{FFFD}\UTF{FFFD}??????\UTF{FFFD}\UTF{FFFD}??\UTF{FFFD}\UTF{FFFD}????\UTF{FFFD}\UTF{FFFD}??\UTF{FFFD}\UTF{FFFD}??????????\UTF{FFFD}\UTF{FFFD}??\UTF{FFFD}\UTF{FFFD}????\UTF{FFFD}\UTF{FFFD}??\UTF{FFFD}\UTF{FFFD}??????\UTF{FFFD}\UTF{FFFD}??\UTF{FFFD}\UTF{FFFD}????\UTF{FFFD}\UTF{FFFD}??\UTF{FFFD}\UTF{FFFD}????????\UTF{FFFD}\UTF{FFFD}??\UTF{FFFD}\UTF{FFFD}????\UTF{FFFD}\UTF{FFFD}??\UTF{FFFD}\UTF{FFFD}??????\UTF{FFFD}\UTF{FFFD}??\UTF{FFFD}\UTF{FFFD}????\UTF{FFFD}\UTF{FFFD}??\UTF{FFFD}\UTF{FFFD}??????????????\UTF{FFFD}\UTF{FFFD}??\UTF{FFFD}\UTF{FFFD}????\UTF{FFFD}\UTF{FFFD}??\UTF{FFFD}\UTF{FFFD}??????\UTF{FFFD}\UTF{FFFD}??\UTF{FFFD}\UTF{FFFD}????\UTF{FFFD}\UTF{FFFD}??\UTF{FFFD}\UTF{FFFD}????????\UTF{FFFD}\UTF{FFFD}??\UTF{FFFD}\UTF{FFFD}????\UTF{FFFD}\UTF{FFFD}??\UTF{FFFD}\UTF{FFFD}??????\UTF{FFFD}\UTF{FFFD}??\UTF{FFFD}\UTF{FFFD}????\UTF{FFFD}\UTF{FFFD}??\UTF{FFFD}\UTF{FFFD}??????????\UTF{FFFD}\UTF{FFFD}??\UTF{FFFD}\UTF{FFFD}????\UTF{FFFD}\UTF{FFFD}??\UTF{FFFD}\UTF{FFFD}??????\UTF{FFFD}\UTF{FFFD}??\UTF{FFFD}\UTF{FFFD}????\UTF{FFFD}\UTF{FFFD}??\UTF{FFFD}\UTF{FFFD}????????\UTF{FFFD}\UTF{FFFD}??\UTF{FFFD}\UTF{FFFD}????\UTF{FFFD}\UTF{FFFD}??\UTF{FFFD}\UTF{FFFD}??????\UTF{FFFD}\UTF{FFFD}??\UTF{FFFD}\UTF{FFFD}????\UTF{FFFD}\UTF{FFFD}??\UTF{FFFD}\UTF{FFFD}????????????\UTF{FFFD}\UTF{FFFD}??\UTF{FFFD}\UTF{FFFD}????\UTF{FFFD}\UTF{FFFD}??\UTF{FFFD}\UTF{FFFD}??????\UTF{FFFD}\UTF{FFFD}??\UTF{FFFD}\UTF{FFFD}????\UTF{FFFD}\UTF{FFFD}??\UTF{FFFD}\UTF{FFFD}????????\UTF{FFFD}\UTF{FFFD}??\UTF{FFFD}\UTF{FFFD}????\UTF{FFFD}\UTF{FFFD}??\UTF{FFFD}\UTF{FFFD}??????\UTF{FFFD}\UTF{FFFD}??\UTF{FFFD}\UTF{FFFD}????\UTF{FFFD}\UTF{FFFD}??\UTF{FFFD}\UTF{FFFD}??????????\UTF{FFFD}\UTF{FFFD}??\UTF{FFFD}\UTF{FFFD}????\UTF{FFFD}\UTF{FFFD}??\UTF{FFFD}\UTF{FFFD}??????\UTF{FFFD}\UTF{FFFD}??\UTF{FFFD}\UTF{FFFD}????\UTF{FFFD}\UTF{FFFD}??\UTF{FFFD}\UTF{FFFD}????????\UTF{FFFD}\UTF{FFFD}??\UTF{FFFD}\UTF{FFFD}????\UTF{FFFD}\UTF{FFFD}??\UTF{FFFD}\UTF{FFFD}??????\UTF{FFFD}\UTF{FFFD}??\UTF{FFFD}\UTF{FFFD}????\UTF{FFFD}\UTF{FFFD}??\UTF{FFFD}\UTF{FFFD}???????????????\UTF{FFFD}\UTF{FFFD}??\UTF{FFFD}\UTF{FFFD}????\UTF{FFFD}\UTF{FFFD}??\UTF{FFFD}\UTF{FFFD}??????\UTF{FFFD}\UTF{FFFD}??\UTF{FFFD}\UTF{FFFD}????\UTF{FFFD}\UTF{FFFD}??\UTF{FFFD}\UTF{FFFD}????????\UTF{FFFD}\UTF{FFFD}??\UTF{FFFD}\UTF{FFFD}????\UTF{FFFD}\UTF{FFFD}??\UTF{FFFD}\UTF{FFFD}??????\UTF{FFFD}\UTF{FFFD}??\UTF{FFFD}\UTF{FFFD}????\UTF{FFFD}\UTF{FFFD}??\UTF{FFFD}\UTF{FFFD}??????????\UTF{FFFD}\UTF{FFFD}??\UTF{FFFD}\UTF{FFFD}????\UTF{FFFD}\UTF{FFFD}??\UTF{FFFD}\UTF{FFFD}??????\UTF{FFFD}\UTF{FFFD}??\UTF{FFFD}\UTF{FFFD}????\UTF{FFFD}\UTF{FFFD}??\UTF{FFFD}\UTF{FFFD}????????\UTF{FFFD}\UTF{FFFD}??\UTF{FFFD}\UTF{FFFD}????\UTF{FFFD}\UTF{FFFD}??\UTF{FFFD}\UTF{FFFD}??????\UTF{FFFD}\UTF{FFFD}??\UTF{FFFD}\UTF{FFFD}????\UTF{FFFD}\UTF{FFFD}??\UTF{FFFD}\UTF{FFFD}????????????\UTF{FFFD}\UTF{FFFD}??\UTF{FFFD}\UTF{FFFD}????\UTF{FFFD}\UTF{FFFD}??\UTF{FFFD}\UTF{FFFD}??????\UTF{FFFD}\UTF{FFFD}??\UTF{FFFD}\UTF{FFFD}????\UTF{FFFD}\UTF{FFFD}??\UTF{FFFD}\UTF{FFFD}????????\UTF{FFFD}\UTF{FFFD}??\UTF{FFFD}\UTF{FFFD}????\UTF{FFFD}\UTF{FFFD}??\UTF{FFFD}\UTF{FFFD}??????\UTF{FFFD}\UTF{FFFD}??\UTF{FFFD}\UTF{FFFD}????\UTF{FFFD}\UTF{FFFD}??\UTF{FFFD}\UTF{FFFD}??????????\UTF{FFFD}\UTF{FFFD}??\UTF{FFFD}\UTF{FFFD}????\UTF{FFFD}\UTF{FFFD}??\UTF{FFFD}\UTF{FFFD}??????\UTF{FFFD}\UTF{FFFD}??\UTF{FFFD}\UTF{FFFD}????\UTF{FFFD}\UTF{FFFD}??\UTF{FFFD}\UTF{FFFD}????????\UTF{FFFD}\UTF{FFFD}??\UTF{FFFD}\UTF{FFFD}????\UTF{FFFD}\UTF{FFFD}??\UTF{FFFD}\UTF{FFFD}??????\UTF{FFFD}\UTF{FFFD}??\UTF{FFFD}\UTF{FFFD}????\UTF{FFFD}\UTF{FFFD}??\UTF{FFFD}\UTF{FFFD}??????????????\UTF{FFFD}\UTF{FFFD}??\UTF{FFFD}\UTF{FFFD}????\UTF{FFFD}\UTF{FFFD}??\UTF{FFFD}\UTF{FFFD}??????\UTF{FFFD}\UTF{FFFD}??\UTF{FFFD}\UTF{FFFD}????\UTF{FFFD}\UTF{FFFD}??\UTF{FFFD}\UTF{FFFD}????????\UTF{FFFD}\UTF{FFFD}??\UTF{FFFD}\UTF{FFFD}????\UTF{FFFD}\UTF{FFFD}??\UTF{FFFD}\UTF{FFFD}??????\UTF{FFFD}\UTF{FFFD}??\UTF{FFFD}\UTF{FFFD}????\UTF{FFFD}\UTF{FFFD}??\UTF{FFFD}\UTF{FFFD}??????????\UTF{FFFD}\UTF{FFFD}??\UTF{FFFD}\UTF{FFFD}????\UTF{FFFD}\UTF{FFFD}??\UTF{FFFD}\UTF{FFFD}??????\UTF{FFFD}\UTF{FFFD}??\UTF{FFFD}\UTF{FFFD}????\UTF{FFFD}\UTF{FFFD}??\UTF{FFFD}\UTF{FFFD}????????\UTF{FFFD}\UTF{FFFD}??\UTF{FFFD}\UTF{FFFD}????\UTF{FFFD}\UTF{FFFD}??\UTF{FFFD}\UTF{FFFD}??????\UTF{FFFD}\UTF{FFFD}??\UTF{FFFD}\UTF{FFFD}????\UTF{FFFD}\UTF{FFFD}??\UTF{FFFD}\UTF{FFFD}????????????\UTF{FFFD}\UTF{FFFD}??\UTF{FFFD}\UTF{FFFD}????\UTF{FFFD}\UTF{FFFD}??\UTF{FFFD}\UTF{FFFD}??????\UTF{FFFD}\UTF{FFFD}??\UTF{FFFD}\UTF{FFFD}????\UTF{FFFD}\UTF{FFFD}??\UTF{FFFD}\UTF{FFFD}????????\UTF{FFFD}\UTF{FFFD}??\UTF{FFFD}\UTF{FFFD}????\UTF{FFFD}\UTF{FFFD}??\UTF{FFFD}\UTF{FFFD}??????\UTF{FFFD}\UTF{FFFD}??\UTF{FFFD}\UTF{FFFD}????\UTF{FFFD}\UTF{FFFD}??\UTF{FFFD}\UTF{FFFD}??????????\UTF{FFFD}\UTF{FFFD}??\UTF{FFFD}\UTF{FFFD}????\UTF{FFFD}\UTF{FFFD}??\UTF{FFFD}\UTF{FFFD}??????\UTF{FFFD}\UTF{FFFD}??\UTF{FFFD}\UTF{FFFD}????\UTF{FFFD}\UTF{FFFD}??\UTF{FFFD}\UTF{FFFD}????????\UTF{FFFD}\UTF{FFFD}??\UTF{FFFD}\UTF{FFFD}????\UTF{FFFD}\UTF{FFFD}??\UTF{FFFD}\UTF{FFFD}??????\UTF{FFFD}\UTF{FFFD}??\UTF{FFFD}\UTF{FFFD}????\UTF{FFFD}\UTF{FFFD}??\UTF{FFFD}\UTF{FFFD}????????????????????\UTF{FFFD}\UTF{FFFD}??\UTF{FFFD}\UTF{FFFD}????\UTF{FFFD}\UTF{FFFD}??\UTF{FFFD}\UTF{FFFD}??????\UTF{FFFD}\UTF{FFFD}??\UTF{FFFD}\UTF{FFFD}????\UTF{FFFD}\UTF{FFFD}??\UTF{FFFD}\UTF{FFFD}????????\UTF{FFFD}\UTF{FFFD}??\UTF{FFFD}\UTF{FFFD}????\UTF{FFFD}\UTF{FFFD}??\UTF{FFFD}\UTF{FFFD}??????\UTF{FFFD}\UTF{FFFD}??\UTF{FFFD}\UTF{FFFD}????\UTF{FFFD}\UTF{FFFD}??\UTF{FFFD}\UTF{FFFD}??????????\UTF{FFFD}\UTF{FFFD}??\UTF{FFFD}\UTF{FFFD}????\UTF{FFFD}\UTF{FFFD}??\UTF{FFFD}\UTF{FFFD}??????\UTF{FFFD}\UTF{FFFD}??\UTF{FFFD}\UTF{FFFD}????\UTF{FFFD}\UTF{FFFD}??\UTF{FFFD}\UTF{FFFD}????????\UTF{FFFD}\UTF{FFFD}??\UTF{FFFD}\UTF{FFFD}????\UTF{FFFD}\UTF{FFFD}??\UTF{FFFD}\UTF{FFFD}??????\UTF{FFFD}\UTF{FFFD}??\UTF{FFFD}\UTF{FFFD}????\UTF{FFFD}\UTF{FFFD}??\UTF{FFFD}\UTF{FFFD}????????????\UTF{FFFD}\UTF{FFFD}??\UTF{FFFD}\UTF{FFFD}????\UTF{FFFD}\UTF{FFFD}??\UTF{FFFD}\UTF{FFFD}??????\UTF{FFFD}\UTF{FFFD}??\UTF{FFFD}\UTF{FFFD}????\UTF{FFFD}\UTF{FFFD}??\UTF{FFFD}\UTF{FFFD}????????\UTF{FFFD}\UTF{FFFD}??\UTF{FFFD}\UTF{FFFD}????\UTF{FFFD}\UTF{FFFD}??\UTF{FFFD}\UTF{FFFD}??????\UTF{FFFD}\UTF{FFFD}??\UTF{FFFD}\UTF{FFFD}????\UTF{FFFD}\UTF{FFFD}??\UTF{FFFD}\UTF{FFFD}??????????\UTF{FFFD}\UTF{FFFD}??\UTF{FFFD}\UTF{FFFD}????\UTF{FFFD}\UTF{FFFD}??\UTF{FFFD}\UTF{FFFD}??????\UTF{FFFD}\UTF{FFFD}??\UTF{FFFD}\UTF{FFFD}????\UTF{FFFD}\UTF{FFFD}??\UTF{FFFD}\UTF{FFFD}????????\UTF{FFFD}\UTF{FFFD}??\UTF{FFFD}\UTF{FFFD}????\UTF{FFFD}\UTF{FFFD}??\UTF{FFFD}\UTF{FFFD}??????\UTF{FFFD}\UTF{FFFD}??\UTF{FFFD}\UTF{FFFD}????\UTF{FFFD}\UTF{FFFD}??\UTF{FFFD}\UTF{FFFD}??????????????\UTF{FFFD}\UTF{FFFD}??\UTF{FFFD}\UTF{FFFD}????\UTF{FFFD}\UTF{FFFD}??\UTF{FFFD}\UTF{FFFD}??????\UTF{FFFD}\UTF{FFFD}??\UTF{FFFD}\UTF{FFFD}????\UTF{FFFD}\UTF{FFFD}??\UTF{FFFD}\UTF{FFFD}????????\UTF{FFFD}\UTF{FFFD}??\UTF{FFFD}\UTF{FFFD}????\UTF{FFFD}\UTF{FFFD}??\UTF{FFFD}\UTF{FFFD}??????\UTF{FFFD}\UTF{FFFD}??\UTF{FFFD}\UTF{FFFD}????\UTF{FFFD}\UTF{FFFD}??\UTF{FFFD}\UTF{FFFD}??????????\UTF{FFFD}\UTF{FFFD}??\UTF{FFFD}\UTF{FFFD}????\UTF{FFFD}\UTF{FFFD}??\UTF{FFFD}\UTF{FFFD}??????\UTF{FFFD}\UTF{FFFD}??\UTF{FFFD}\UTF{FFFD}????\UTF{FFFD}\UTF{FFFD}??\UTF{FFFD}\UTF{FFFD}????????\UTF{FFFD}\UTF{FFFD}??\UTF{FFFD}\UTF{FFFD}????\UTF{FFFD}\UTF{FFFD}??\UTF{FFFD}\UTF{FFFD}??????\UTF{FFFD}\UTF{FFFD}??\UTF{FFFD}\UTF{FFFD}????\UTF{FFFD}\UTF{FFFD}??\UTF{FFFD}\UTF{FFFD}????????????\UTF{FFFD}\UTF{FFFD}??\UTF{FFFD}\UTF{FFFD}????\UTF{FFFD}\UTF{FFFD}??\UTF{FFFD}\UTF{FFFD}??????\UTF{FFFD}\UTF{FFFD}??\UTF{FFFD}\UTF{FFFD}????\UTF{FFFD}\UTF{FFFD}??\UTF{FFFD}\UTF{FFFD}????????\UTF{FFFD}\UTF{FFFD}??\UTF{FFFD}\UTF{FFFD}????\UTF{FFFD}\UTF{FFFD}??\UTF{FFFD}\UTF{FFFD}??????\UTF{FFFD}\UTF{FFFD}??\UTF{FFFD}\UTF{FFFD}????\UTF{FFFD}\UTF{FFFD}??\UTF{FFFD}\UTF{FFFD}??????????\UTF{FFFD}\UTF{FFFD}??\UTF{FFFD}\UTF{FFFD}????\UTF{FFFD}\UTF{FFFD}??\UTF{FFFD}\UTF{FFFD}??????\UTF{FFFD}\UTF{FFFD}??\UTF{FFFD}\UTF{FFFD}????\UTF{FFFD}\UTF{FFFD}??\UTF{FFFD}\UTF{FFFD}????????\UTF{FFFD}\UTF{FFFD}??\UTF{FFFD}\UTF{FFFD}????\UTF{FFFD}\UTF{FFFD}??\UTF{FFFD}\UTF{FFFD}??????\UTF{FFFD}\UTF{FFFD}??\UTF{FFFD}\UTF{FFFD}????\UTF{FFFD}\UTF{FFFD}??\UTF{FFFD}\UTF{FFFD}??????????????????????\UTF{FFFD}\UTF{FFFD}??\UTF{FFFD}\UTF{FFFD}????\UTF{FFFD}\UTF{FFFD}??\UTF{FFFD}\UTF{FFFD}??????\UTF{FFFD}\UTF{FFFD}??\UTF{FFFD}\UTF{FFFD}????\UTF{FFFD}\UTF{FFFD}??\UTF{FFFD}\UTF{FFFD}????????\UTF{FFFD}\UTF{FFFD}??\UTF{FFFD}\UTF{FFFD}????\UTF{FFFD}\UTF{FFFD}??\UTF{FFFD}\UTF{FFFD}??????\UTF{FFFD}\UTF{FFFD}??\UTF{FFFD}\UTF{FFFD}????\UTF{FFFD}\UTF{FFFD}??\UTF{FFFD}\UTF{FFFD}??????????\UTF{FFFD}\UTF{FFFD}??\UTF{FFFD}\UTF{FFFD}????\UTF{FFFD}\UTF{FFFD}??\UTF{FFFD}\UTF{FFFD}??????\UTF{FFFD}\UTF{FFFD}??\UTF{FFFD}\UTF{FFFD}????\UTF{FFFD}\UTF{FFFD}??\UTF{FFFD}\UTF{FFFD}????????\UTF{FFFD}\UTF{FFFD}??\UTF{FFFD}\UTF{FFFD}????\UTF{FFFD}\UTF{FFFD}??\UTF{FFFD}\UTF{FFFD}??????\UTF{FFFD}\UTF{FFFD}??\UTF{FFFD}\UTF{FFFD}????\UTF{FFFD}\UTF{FFFD}??\UTF{FFFD}\UTF{FFFD}????????????\UTF{FFFD}\UTF{FFFD}??\UTF{FFFD}\UTF{FFFD}????\UTF{FFFD}\UTF{FFFD}??\UTF{FFFD}\UTF{FFFD}??????\UTF{FFFD}\UTF{FFFD}??\UTF{FFFD}\UTF{FFFD}????\UTF{FFFD}\UTF{FFFD}??\UTF{FFFD}\UTF{FFFD}????????\UTF{FFFD}\UTF{FFFD}??\UTF{FFFD}\UTF{FFFD}????\UTF{FFFD}\UTF{FFFD}??\UTF{FFFD}\UTF{FFFD}??????\UTF{FFFD}\UTF{FFFD}??\UTF{FFFD}\UTF{FFFD}????\UTF{FFFD}\UTF{FFFD}??\UTF{FFFD}\UTF{FFFD}??????????\UTF{FFFD}\UTF{FFFD}??\UTF{FFFD}\UTF{FFFD}????\UTF{FFFD}\UTF{FFFD}??\UTF{FFFD}\UTF{FFFD}??????\UTF{FFFD}\UTF{FFFD}??\UTF{FFFD}\UTF{FFFD}????\UTF{FFFD}\UTF{FFFD}??\UTF{FFFD}\UTF{FFFD}????????\UTF{FFFD}\UTF{FFFD}??\UTF{FFFD}\UTF{FFFD}????\UTF{FFFD}\UTF{FFFD}??\UTF{FFFD}\UTF{FFFD}??????\UTF{FFFD}\UTF{FFFD}??\UTF{FFFD}\UTF{FFFD}????\UTF{FFFD}\UTF{FFFD}??\UTF{FFFD}\UTF{FFFD}??????????????\UTF{FFFD}\UTF{FFFD}??\UTF{FFFD}\UTF{FFFD}????\UTF{FFFD}\UTF{FFFD}??\UTF{FFFD}\UTF{FFFD}??????\UTF{FFFD}\UTF{FFFD}??\UTF{FFFD}\UTF{FFFD}????\UTF{FFFD}\UTF{FFFD}??\UTF{FFFD}\UTF{FFFD}????????\UTF{FFFD}\UTF{FFFD}??\UTF{FFFD}\UTF{FFFD}????\UTF{FFFD}\UTF{FFFD}??\UTF{FFFD}\UTF{FFFD}??????\UTF{FFFD}\UTF{FFFD}??\UTF{FFFD}\UTF{FFFD}????\UTF{FFFD}\UTF{FFFD}??\UTF{FFFD}\UTF{FFFD}??????????\UTF{FFFD}\UTF{FFFD}??\UTF{FFFD}\UTF{FFFD}????\UTF{FFFD}\UTF{FFFD}??\UTF{FFFD}\UTF{FFFD}??????\UTF{FFFD}\UTF{FFFD}??\UTF{FFFD}\UTF{FFFD}????\UTF{FFFD}\UTF{FFFD}??\UTF{FFFD}\UTF{FFFD}????????\UTF{FFFD}\UTF{FFFD}??\UTF{FFFD}\UTF{FFFD}????\UTF{FFFD}\UTF{FFFD}??\UTF{FFFD}\UTF{FFFD}??????\UTF{FFFD}\UTF{FFFD}??\UTF{FFFD}\UTF{FFFD}????\UTF{FFFD}\UTF{FFFD}??\UTF{FFFD}\UTF{FFFD}????????????\UTF{FFFD}\UTF{FFFD}??\UTF{FFFD}\UTF{FFFD}????\UTF{FFFD}\UTF{FFFD}??\UTF{FFFD}\UTF{FFFD}??????\UTF{FFFD}\UTF{FFFD}??\UTF{FFFD}\UTF{FFFD}????\UTF{FFFD}\UTF{FFFD}??\UTF{FFFD}\UTF{FFFD}????????\UTF{FFFD}\UTF{FFFD}??\UTF{FFFD}\UTF{FFFD}????\UTF{FFFD}\UTF{FFFD}??\UTF{FFFD}\UTF{FFFD}??????\UTF{FFFD}\UTF{FFFD}??\UTF{FFFD}\UTF{FFFD}????\UTF{FFFD}\UTF{FFFD}??\UTF{FFFD}\UTF{FFFD}??????????\UTF{FFFD}\UTF{FFFD}??\UTF{FFFD}\UTF{FFFD}????\UTF{FFFD}\UTF{FFFD}??\UTF{FFFD}\UTF{FFFD}??????\UTF{FFFD}\UTF{FFFD}??\UTF{FFFD}\UTF{FFFD}????\UTF{FFFD}\UTF{FFFD}??\UTF{FFFD}\UTF{FFFD}????????\UTF{FFFD}\UTF{FFFD}??\UTF{FFFD}\UTF{FFFD}????\UTF{FFFD}\UTF{FFFD}??\UTF{FFFD}\UTF{FFFD}??????\UTF{FFFD}\UTF{FFFD}??\UTF{FFFD}\UTF{FFFD}????\UTF{FFFD}\UTF{FFFD}??\UTF{FFFD}\UTF{FFFD}????????????????\UTF{FFFD}\UTF{FFFD}??\UTF{FFFD}\UTF{FFFD}????\UTF{FFFD}\UTF{FFFD}??\UTF{FFFD}\UTF{FFFD}??????\UTF{FFFD}\UTF{FFFD}??\UTF{FFFD}\UTF{FFFD}????\UTF{FFFD}\UTF{FFFD}??\UTF{FFFD}\UTF{FFFD}????????\UTF{FFFD}\UTF{FFFD}??\UTF{FFFD}\UTF{FFFD}????\UTF{FFFD}\UTF{FFFD}??\UTF{FFFD}\UTF{FFFD}??????\UTF{FFFD}\UTF{FFFD}??\UTF{FFFD}\UTF{FFFD}????\UTF{FFFD}\UTF{FFFD}??\UTF{FFFD}\UTF{FFFD}??????????\UTF{FFFD}\UTF{FFFD}??\UTF{FFFD}\UTF{FFFD}????\UTF{FFFD}\UTF{FFFD}??\UTF{FFFD}\UTF{FFFD}??????\UTF{FFFD}\UTF{FFFD}??\UTF{FFFD}\UTF{FFFD}????\UTF{FFFD}\UTF{FFFD}??\UTF{FFFD}\UTF{FFFD}????????\UTF{FFFD}\UTF{FFFD}??\UTF{FFFD}\UTF{FFFD}????\UTF{FFFD}\UTF{FFFD}??\UTF{FFFD}\UTF{FFFD}??????\UTF{FFFD}\UTF{FFFD}??\UTF{FFFD}\UTF{FFFD}????\UTF{FFFD}\UTF{FFFD}??\UTF{FFFD}\UTF{FFFD}????????????\UTF{FFFD}\UTF{FFFD}??\UTF{FFFD}\UTF{FFFD}????\UTF{FFFD}\UTF{FFFD}??\UTF{FFFD}\UTF{FFFD}??????\UTF{FFFD}\UTF{FFFD}??\UTF{FFFD}\UTF{FFFD}????\UTF{FFFD}\UTF{FFFD}??\UTF{FFFD}\UTF{FFFD}????????\UTF{FFFD}\UTF{FFFD}??\UTF{FFFD}\UTF{FFFD}????\UTF{FFFD}\UTF{FFFD}??\UTF{FFFD}\UTF{FFFD}??????\UTF{FFFD}\UTF{FFFD}??\UTF{FFFD}\UTF{FFFD}????\UTF{FFFD}\UTF{FFFD}??\UTF{FFFD}\UTF{FFFD}??????????\UTF{FFFD}\UTF{FFFD}??\UTF{FFFD}\UTF{FFFD}????\UTF{FFFD}\UTF{FFFD}??\UTF{FFFD}\UTF{FFFD}??????\UTF{FFFD}\UTF{FFFD}??\UTF{FFFD}\UTF{FFFD}????\UTF{FFFD}\UTF{FFFD}??\UTF{FFFD}\UTF{FFFD}????????\UTF{FFFD}\UTF{FFFD}??\UTF{FFFD}\UTF{FFFD}????\UTF{FFFD}\UTF{FFFD}??\UTF{FFFD}\UTF{FFFD}??????\UTF{FFFD}\UTF{FFFD}??\UTF{FFFD}\UTF{FFFD}????\UTF{FFFD}\UTF{FFFD}??\UTF{FFFD}\UTF{FFFD}??????????????\UTF{FFFD}\UTF{FFFD}??\UTF{FFFD}\UTF{FFFD}????\UTF{FFFD}\UTF{FFFD}??\UTF{FFFD}\UTF{FFFD}??????\UTF{FFFD}\UTF{FFFD}??\UTF{FFFD}\UTF{FFFD}????\UTF{FFFD}\UTF{FFFD}??\UTF{FFFD}\UTF{FFFD}????????\UTF{FFFD}\UTF{FFFD}??\UTF{FFFD}\UTF{FFFD}????\UTF{FFFD}\UTF{FFFD}??\UTF{FFFD}\UTF{FFFD}??????\UTF{FFFD}\UTF{FFFD}??\UTF{FFFD}\UTF{FFFD}????\UTF{FFFD}\UTF{FFFD}??\UTF{FFFD}\UTF{FFFD}??????????\UTF{FFFD}\UTF{FFFD}??\UTF{FFFD}\UTF{FFFD}????\UTF{FFFD}\UTF{FFFD}??\UTF{FFFD}\UTF{FFFD}??????\UTF{FFFD}\UTF{FFFD}??\UTF{FFFD}\UTF{FFFD}????\UTF{FFFD}\UTF{FFFD}??\UTF{FFFD}\UTF{FFFD}????????\UTF{FFFD}\UTF{FFFD}??\UTF{FFFD}\UTF{FFFD}????\UTF{FFFD}\UTF{FFFD}??\UTF{FFFD}\UTF{FFFD}??????\UTF{FFFD}\UTF{FFFD}??\UTF{FFFD}\UTF{FFFD}????\UTF{FFFD}\UTF{FFFD}??\UTF{FFFD}\UTF{FFFD}????????????\UTF{FFFD}\UTF{FFFD}??\UTF{FFFD}\UTF{FFFD}????\UTF{FFFD}\UTF{FFFD}??\UTF{FFFD}\UTF{FFFD}??????\UTF{FFFD}\UTF{FFFD}??\UTF{FFFD}\UTF{FFFD}????\UTF{FFFD}\UTF{FFFD}??\UTF{FFFD}\UTF{FFFD}????????\UTF{FFFD}\UTF{FFFD}??\UTF{FFFD}\UTF{FFFD}????\UTF{FFFD}\UTF{FFFD}??\UTF{FFFD}\UTF{FFFD}??????\UTF{FFFD}\UTF{FFFD}??\UTF{FFFD}\UTF{FFFD}????\UTF{FFFD}\UTF{FFFD}??\UTF{FFFD}\UTF{FFFD}??????????\UTF{FFFD}\UTF{FFFD}??\UTF{FFFD}\UTF{FFFD}????\UTF{FFFD}\UTF{FFFD}??\UTF{FFFD}\UTF{FFFD}??????\UTF{FFFD}\UTF{FFFD}??\UTF{FFFD}\UTF{FFFD}????\UTF{FFFD}\UTF{FFFD}??\UTF{FFFD}\UTF{FFFD}????????\UTF{FFFD}\UTF{FFFD}??\UTF{FFFD}\UTF{FFFD}????\UTF{FFFD}\UTF{FFFD}??\UTF{FFFD}\UTF{FFFD}??????\UTF{FFFD}\UTF{FFFD}??\UTF{FFFD}\UTF{FFFD}????\UTF{FFFD}\UTF{FFFD}??\UTF{FFFD}\UTF{FFFD}?????????????????????????\UTF{FFFD}\UTF{FFFD}??\UTF{FFFD}\UTF{FFFD}????\UTF{FFFD}\UTF{FFFD}??\UTF{FFFD}\UTF{FFFD}??????\UTF{FFFD}\UTF{FFFD}??\UTF{FFFD}\UTF{FFFD}????\UTF{FFFD}\UTF{FFFD}??\UTF{FFFD}\UTF{FFFD}????????\UTF{FFFD}\UTF{FFFD}??\UTF{FFFD}\UTF{FFFD}????\UTF{FFFD}\UTF{FFFD}??\UTF{FFFD}\UTF{FFFD}??????\UTF{FFFD}\UTF{FFFD}??\UTF{FFFD}\UTF{FFFD}????\UTF{FFFD}\UTF{FFFD}??\UTF{FFFD}\UTF{FFFD}??????????\UTF{FFFD}\UTF{FFFD}??\UTF{FFFD}\UTF{FFFD}????\UTF{FFFD}\UTF{FFFD}??\UTF{FFFD}\UTF{FFFD}??????\UTF{FFFD}\UTF{FFFD}??\UTF{FFFD}\UTF{FFFD}????\UTF{FFFD}\UTF{FFFD}??\UTF{FFFD}\UTF{FFFD}????????\UTF{FFFD}\UTF{FFFD}??\UTF{FFFD}\UTF{FFFD}????\UTF{FFFD}\UTF{FFFD}??\UTF{FFFD}\UTF{FFFD}??????\UTF{FFFD}\UTF{FFFD}??\UTF{FFFD}\UTF{FFFD}????\UTF{FFFD}\UTF{FFFD}??\UTF{FFFD}\UTF{FFFD}????????????\UTF{FFFD}\UTF{FFFD}??\UTF{FFFD}\UTF{FFFD}????\UTF{FFFD}\UTF{FFFD}??\UTF{FFFD}\UTF{FFFD}??????\UTF{FFFD}\UTF{FFFD}??\UTF{FFFD}\UTF{FFFD}????\UTF{FFFD}\UTF{FFFD}??\UTF{FFFD}\UTF{FFFD}????????\UTF{FFFD}\UTF{FFFD}??\UTF{FFFD}\UTF{FFFD}????\UTF{FFFD}\UTF{FFFD}??\UTF{FFFD}\UTF{FFFD}??????\UTF{FFFD}\UTF{FFFD}??\UTF{FFFD}\UTF{FFFD}????\UTF{FFFD}\UTF{FFFD}??\UTF{FFFD}\UTF{FFFD}??????????\UTF{FFFD}\UTF{FFFD}??\UTF{FFFD}\UTF{FFFD}????\UTF{FFFD}\UTF{FFFD}??\UTF{FFFD}\UTF{FFFD}??????\UTF{FFFD}\UTF{FFFD}??\UTF{FFFD}\UTF{FFFD}????\UTF{FFFD}\UTF{FFFD}??\UTF{FFFD}\UTF{FFFD}????????\UTF{FFFD}\UTF{FFFD}??\UTF{FFFD}\UTF{FFFD}????\UTF{FFFD}\UTF{FFFD}??\UTF{FFFD}\UTF{FFFD}??????\UTF{FFFD}\UTF{FFFD}??\UTF{FFFD}\UTF{FFFD}????\UTF{FFFD}\UTF{FFFD}??\UTF{FFFD}\UTF{FFFD}}??????
\usepackage{braket}

\theoremstyle{plain}
\newtheorem{theorem}{Theorem}
\newtheorem{Lemma}{Lemma}
\newtheorem{proposition}{Proposition}
\newtheorem{corollary}{Corollary}
\newtheorem{fact}{Fact}

\theoremstyle{definition}
\newtheorem{definition}{Definition}
\newtheorem{notation}{Notation}

\newcommand{\gpv}{\operatorname{GPV}}

\begin{document}
\begin{abstract}
Vassiliev introduced filtered invariants of knots using an unknotting operation, called crossing changes.    Goussarov, Polyak, and Viro introduced other filtered invariants of virtual knots, which order is called $\gpv$-order, using an unknotting operation, called virtualization.  We defined other filtered invariants, which order is called $F$-order, of virtual knots using an unknotting operation, called forbidden moves.  In this paper, we show that the set of virtual knot invariants of $F$-order $\le n+1$ is strictly stronger than that of $F$-order $\le n$ and that of $\gpv$-order $\le 2n+1$.  To obtain the result, we show that the set of virtual knot invariants of $F$-order $\le n$ contains every Goussarov-Polyak-Viro invariant of $\gpv$-order $\le 2n+1$, which implies that the set of virtual knot invariants of $F$-order is a complete invariant of classical and virtual knots.    
\end{abstract}
\author[N.~Ito]{Noboru Ito}
\address{Graduate School of Mathematical Sciences, The University of Tokyo, 3-8-1, Komaba, Meguro-ku, Tokyo, 153-8914, Japan}
\email{noboru@ms.u-tokyo.ac.jp}
\author[M.~Sakurai]{Migiwa Sakurai}
\address{Department of Materials Science and Engineering, College of Engineering, Shibaura Institute of Technology, 307 Fukasaku, Minuma-ku, Saitama-shi, Saitama, 337-8570, Japan}
\email{migiwa@shibaura-it.ac.jp}
\title[Higher-order finite type invariants of classical and virtual knots]{Higher-order finite type invariants of classical and virtual knots and unknotting operations}
\keywords{finite type invariants, knots, virtual knots, unknotting operations,
virtualizations, forbidden moves}

\maketitle
\section{Introduction}
In this paper, we discuss on higher-orders of filtered invariants which are complete invariants of classical and virtual knots and whose filtration was introduced by us \cite{IS} recently.  
In particular, we show that the set of virtual knot invariants of $F$-order $\le n+1$ is strictly stronger than those of $F$-order $\le n$.
Here, we say that a map is a complete invariant if it distinguishes any pair of (virtual) knots.    

Traditionally, a (classical) knot is a smoothly embedded circle in $\mathbb{R}^3$ and a virtual knot is a stable equivalence class of an image of a regular projection to a closed surface $\Sigma$ of an embedded circle in $\Sigma \times I$, where $I$ is an interval that is homeomorphic to $[0, 1]$ (for the definitions of classical knots and virtual knots, see Section~\ref{sec:pre} and for stable equivalences as above, see Carter-Kamada-Saito \cite{CKS}).  It is known that virtually isotopic classical knots are isotopic.  

In 1990, Vassiliev \cite{Va} introduced a filtered invariants by a standard unknotting operation, called \emph{crossing change}. 
For Vassiliev invariants, in 1990, Ohyama \cite{Ohyama1} introduced a notion of $n$-triviality and, in 1992, Taniyama \cite{Taniyama} generalized it to obtain a notion of $n$-similarity.  If a knot $K$ is $n$-similar to another knot $K'$, then the value of Vassiliev invariant of order $\le n-1$ of $K$ equals that of $K'$.  In 1990's, Habiro \cite{habiro} gave a notion corresponding to a generalized notion of $n$-similarity, introduced $C_k$-moves, and showed that the same Vassiliev invariants of order $\le k-1$ if and only if they are transformed into each other by $C_k$-moves.  Goussarov independently proved a similar result \cite{goussarov}.     

In 2000, Goussarov, Polyak, and Viro \cite{GPV} introduced another order and filtration by another unknotting operation, called \emph{virtualization}, for classical and virtual knots and using their theory, then we have another framework to obtain concrete Vassiliev invariants via dual spaces generated by oriented chord diagrams, called arrow diagrams.  We note that the set of invariants of Goussarov, Polyak, and Viro implies a complete invariant of classical and virtual knots.   In this paper, each virtual knot invariants given by Goussarov, Polyak, and Viro is called a {\it finite type invariant of $\gpv$-order}.  
In 2017, we introduced finite type invariants of $F$-order of virtual knots and a notion that corresponds to $n$-similarity, using another unknotting operation forbidden moves, where, in 2001, Kanenobu \cite{kanenobu2}, and independently Nelson \cite{nelson} showed that forbidden moves give an unknotting operation.   
In this paper, we discuss on relationships among finite type invariants of Vassiliev, Goussarov-Polyak-Viro, and us.  
  
A notion of $n$-trivialities is defined in the following.  Let $\mathcal{A}$ be a collection of $n$ pairwise disjoint, nonempty subset that consists of isolated sufficiently small disks on which unknotting operations are applied.  For every subset $T$ of the power set of $\mathcal{A}$, denote a knot diagram by $K(T)$ by applying an unknotting operation in each small disk in $T$.  A knot $K$ is said to be $n$-trivial if there exists $T$ such that $K(\emptyset)$ is a diagram of $K$ and $K(T)$ is a knot diagram of the unknot.  
By replacing the unknotting operation, crossing changes, with virtualization (forbidden moves, resp.), we have a notion of ``$n$-similarity'' corresponding to invariants of Goussarov-Polyak-Viro (our invariant, resp.) that is called $\gpv_n$-similar ($F_n$-similar, resp.).  

In this paper, we give infinitely many pairs of $n$-similar classical knots of each $\gpv$-order $n$.  We also obtain infinitely many pairs of $n$-similar virtual knots of each $F$-order $n$.
We show that if a virtual knot $K$ is $n$-similar to another virtual knot $K'$ by $F$-order, then $K$ is $n$-similar to $K'$ by $\gpv$-order.  

In this paper, we show that the set of invariants of $F$-order $\le n+1$ is strictly stronger than those of $F$-order $\le n$ and that of $\gpv$-order $\le 2n+1$.   To obtain the result, we show that the set of virtual knot invariants of $F$-order $\le n$ contains every Goussarov-Polyak-Viro invariant of $\gpv$-order $\le 2n+1$.  It implies that the set of virtual knot invariants of $F$-order is a complete invariant of classical and virtual knots
because it is known that the set of Goussarov-Polyak-Viro invariants are complete invariants of classical and virtual knots.  

%%%%%%%%%

\section{Preliminaries}
\label{sec:pre}
\subsection{Virtual knots and local moves}
\label{sub:DE}

Suppose that $f:\underbrace{{\mathbb S}^1 \sqcup \cdots \sqcup {\mathbb S}^1}_{\mu}  \rightarrow {\mathbb R}^3$ is a smooth embedding, and the image
$f(\underbrace{{\mathbb S}^1 \sqcup \cdots \sqcup {\mathbb S}^1}_{\mu})$ is called a $\mu$-{\it component link}.
We say that two $\mu$-component link diagrams $L_0$ and $L_1$ are \emph{isotopic}
if there exists an isotopy $h_t : {\mathbb R}^3 \rightarrow {\mathbb R}^3, t \in [0, 1]$, with $h_0 = \operatorname{id}$ and $h_1 (L_0)$ $=$ $L_1$.   
Let $L$ be a $\mu$-component link and $p$ a regular projection ${\mathbb R}^3 \rightarrow {\mathbb R}^2$, i.e., $p(L)$ is generic immersed plane closed curves.  In particular, every self-intersection of the image $p(L)$ is a transverse double point.  Then, the image $p(L) (\subset \mathbb{R}^2)$, up to plane isotopy, with over/under information of each double point is called a \emph{diagram} of $L$.
if $\mu =1$, a $\mu$-component link diagram is called a \emph{knot diagram} or a \emph{classical knot diagram}, and a $\mu$-component link is called a \emph{knot} or \emph{classical knot}.

We define a \emph{$\mu$-component virtual link diagram} as a $\mu$-component link diagram having some virtual crossings (right) as well as real crossings in Fig.~\ref{fig:crossings} (left). \emph{Two $\mu$-component virtual link diagrams are equivalent} if one can be transformed from the other by a finite sequence of \emph{generalized Reidemeister moves} in Fig.~\ref{fig:GRM}.  We call the equivalence class of $\mu$-component virtual link diagrams modulo the generalized Reidemeister moves a {\em $\mu$-component virtual link}. In particular, if $\mu =1$, a $\mu$-component virtual link diagram is called a \emph{virtual knot diagram}, and a $\mu$-component virtual link is called a \emph{virtual knot}.
A virtual knot diagram is pointed if it is endowed with a base point
that is not a double point.
Generalized Reidemeister moves away from the base point imply an equivalence relation for pointed virtual knot diagrams,
and each equivalence class is called a long virtual knot.

%\begin{figure}[h]
%\centering
%\includegraphics[width=7cm]{grm.pdf}
%\end{figure}
\begin{figure}[htbp]
\centering
\includegraphics[width=3cm,clip]{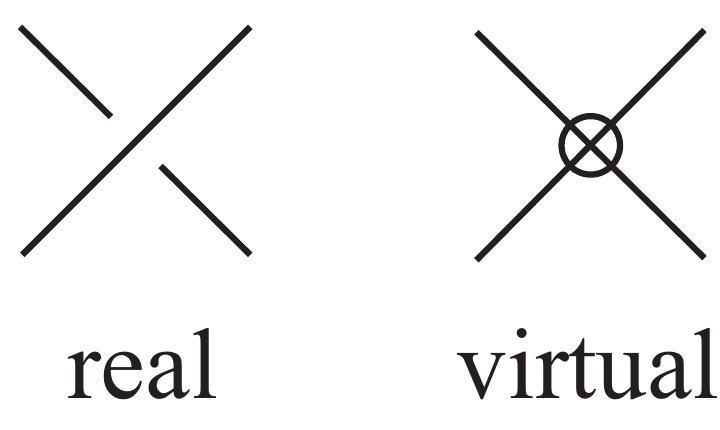}
\caption{Crossing types.}
\label{fig:crossings}
\end{figure}

\begin{figure}[htbp]
\centering
\includegraphics[width=9cm,clip]{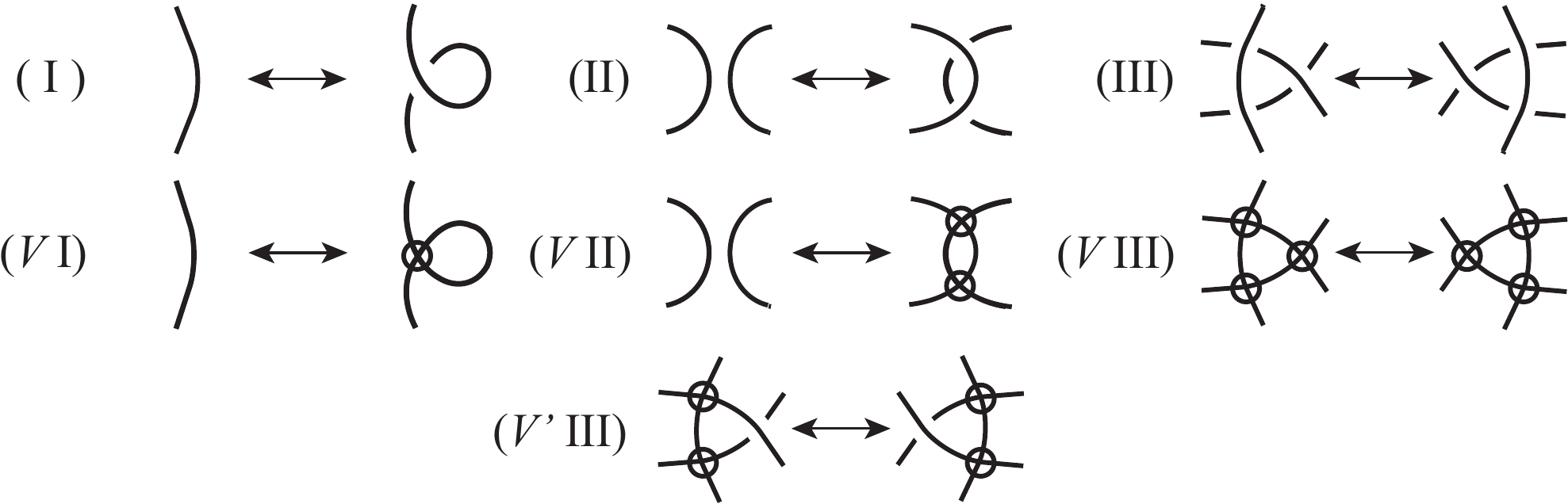}
\caption{Generalized Reidemeister moves.}
\label{fig:GRM}
\end{figure}

By definition, a virtual knot diagram is the image of a generic immersion from $\mathbb S^1$ into $\mathbb R^2$. Let $K$ be a (long) virtual knot and $D_K$ a (long) virtual knot diagram of $K$.  Then, $D_K$ is regarded as the image of a generic immersion $\mathbb{S}^1$ $\to$ $\mathbb{R}^2$.  A (based) {\em Gauss diagram} for $D_K$ is the preimage of $D_K$ with chords, each of which connects the preimages of each real crossing.  We specify over/under information of each real crossing on the corresponding chord by directing the chord toward the under path and assigning each chord with the sign of the crossing (Fig.~\ref{fig:sign}).

It is well-known that there exists a bijection from the set of (long) virtual knots to the set of equivalence classes of their (based) Gauss diagrams modulo the \emph{generalized Reidemeister moves of $($based$)$ Gauss diagrams} as shown in Fig.~\ref{fig:D_GRM}.  We identify a (long) virtual knot with an equivalence class of (based) Gauss diagrams, and we freely use either one of them depending on situations in this paper.

\begin{figure}[htbp]
\centering
\includegraphics[width=3cm,clip]{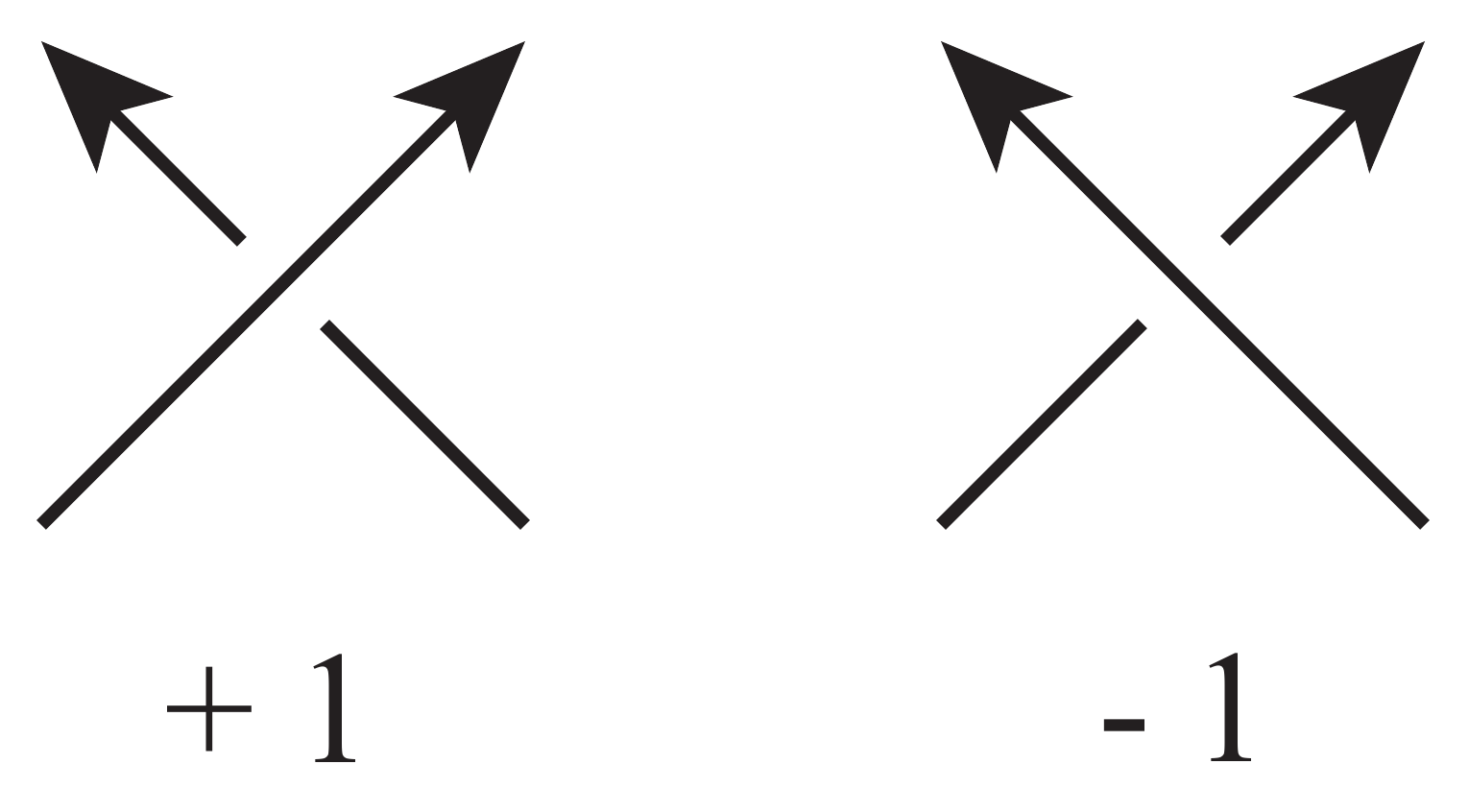}
\caption{The sign of a real crossing.}
\label{fig:sign}
\end{figure}

\begin{figure}[htbp]
\centering
\includegraphics[width=6cm,clip]{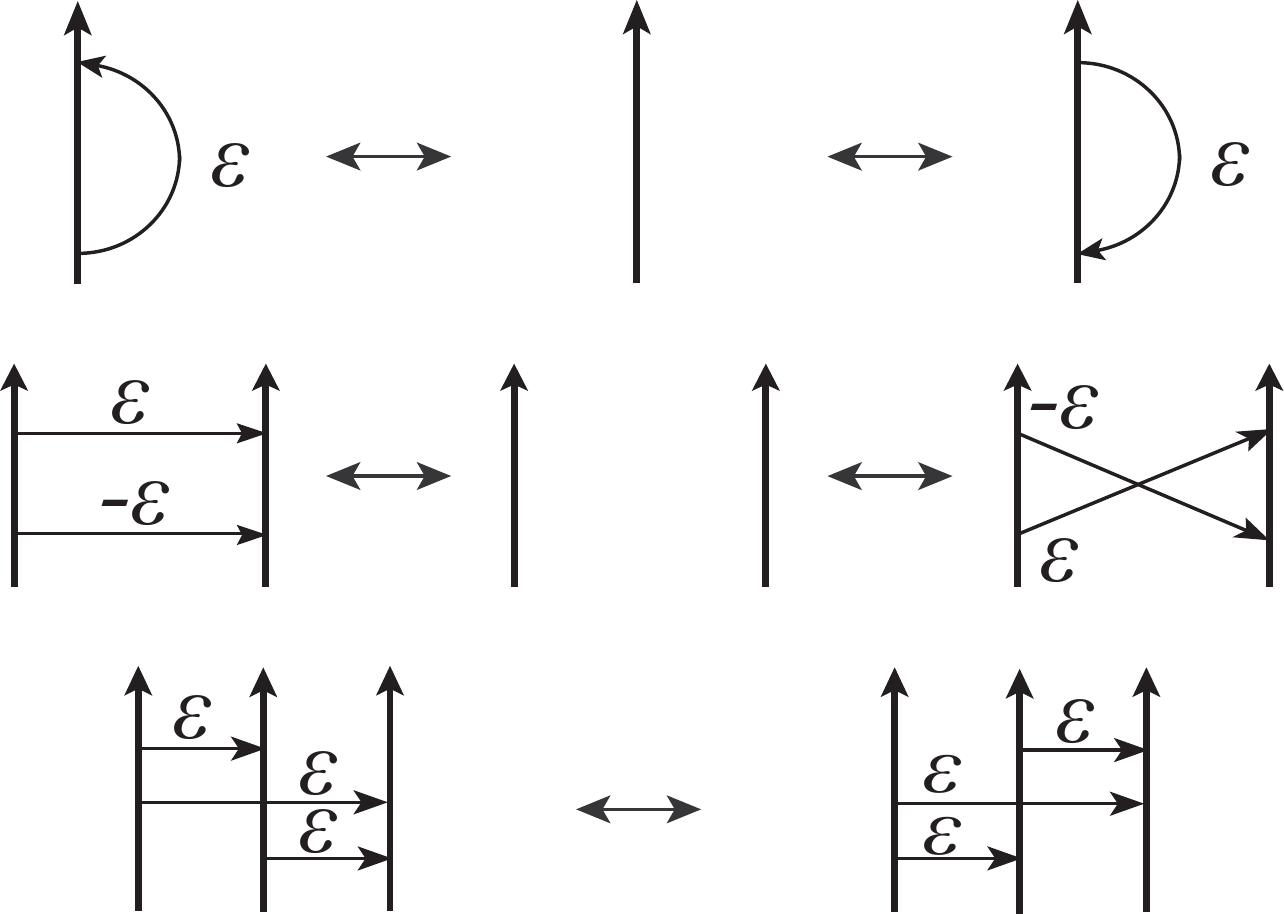}
\caption{Generalized Reidemeister moves of Gauss diagrams for virtual links.}
\label{fig:D_GRM}
\end{figure}

\begin{comment}
\begin{definition}[arrow diagram]
An \emph{arrow diagram} is just a Gauss diagram with all chords drawn dashed. Let $\mathcal A$ be the set of all arrow diagrams, and $\mathcal D$ the set of all Gauss diagrams. A \emph{sub-Gauss diagram} of $D \in \mathcal D$ is a Gauss diagram consisting of a subset of the chords of $D$. Define a map $i : \mathcal D \rightarrow \mathbb{Z} \mathcal A$ by the map that makes all the chords of a Gauss diagram dashed, and define a map $I:{\mathcal D} \rightarrow {\mathbb{Z} \mathcal A}$ by
\[
I\left( D\right) = \sum_{D' \subset D}i(D'),
\]
where the sum is over all subdiagrams of $D$. Extend the map $I$ to $\mathbb{Z} \mathcal D$ linearly. On the generators of $\mathbb{Z} \mathcal A$, define $( D, E )$ to be $1$ if $D = E$ and $0$ otherwise, and then extend $( \cdot , \cdot )$ bilinearly. Put
\begin{equation}\label{eq1}
	\langle  A,D\rangle =\left( A,I\left( D\right) \right),
\end{equation}
for any $D \in \mathcal D$ and $A \in \mathbb{Z} \mathcal A$.  
\end{definition}
\end{comment}

\begin{definition}[Gauss diagram formula]
Let $\mathcal{D}$ be the set of (based) Gauss diagrams and let $D \in \mathcal{D}$.  
A \emph{sub-Gauss diagram} of $D \in \mathcal D$ is a (based) Gauss diagram obtained by ignoring some chords of $D$.  Then, we write $D' \subset D$ to mean $D'$ is a sub-Gauss diagram of $D$.  
Define a map $J : {\mathcal D} \rightarrow \mathbb Z \mathcal D$ by
\[
J\left( D\right) = \sum_{D' \subset D}D',
\]
where the sum is over all sub-Gauss diagrams of $D$. Extend the map $J$ to $\mathbb{Z} \mathcal D$ linearly. 
On the generators of $\mathbb{Z} \mathcal D$, define $( D, E )$ to be $1$ if $D = E$ and $0$ otherwise. Then, extend $( \cdot , \cdot )$ bilinearly. Put
\begin{equation}\label{eq1}
	\langle  A,D\rangle =\left( A, J\left( D\right) \right),
\end{equation}
for any $D \in \mathcal D$ and $A \in \mathbb{Z} \mathcal D$.
The map  $\langle  A,D\rangle$ $(D \in \mathcal D, A \in {\mathbb Z}\mathcal D)$ is called a {\em Gauss diagram formula} for virtual knots.
\end{definition}
   
Let $K$ be a virtual knot and $D_K$ a diagram of $K$.  Suppose that a virtual knot diagram $D_K$ and the virtual knot diagram with no double point are equivalent under generalized Reidemeister moves.  Then, $K$ is called the \emph{unknot}.  Let $D_K$ be a virtual knot diagram.  Then, a \emph{local move} is a replacement of a sufficiently small disk $d (\subset {\mathbb{R}}^2)$ on $D (\subset {\mathbb{R}^2})$ by another disk $d' (\subset {\mathbb{R}^2})$ such that $\partial d$ $=$ $\partial d'$ and $(\mathbb{R}^2 \setminus d) \cup d'$ gives a virtual knot diagram.   
For a virtual knot diagram, the type of local move shown in Fig.~\ref{fig:virtualization} is called {\em virtualization}.  Every local move shown in Fig.~\ref{fig:forbidden} is called a {\em forbidden move}.  Note that a notion ``forbidden move'' consists of the two local moves, as shown in Fig.~\ref{fig:forbidden}.  
Let $M$ be a local move. Suppose that any virtual knot diagram is transformed into the virtual knot diagram with no double points by a finite sequence of local moves belonging to a notion $M$ and generalized Reidemeister moves.  Then, $M$ is called an {\it unknotting operation} for virtual knots. 

\begin{figure}[htbp]
\centering
\includegraphics[width=4cm,clip]{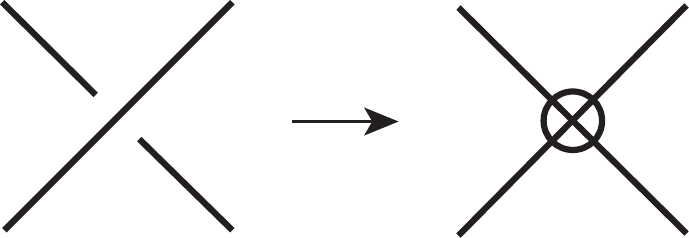}
\caption{Virtualization.}
\label{fig:virtualization}
\end{figure}

\begin{figure}[htbp]
\centering
\includegraphics[width=9.5cm]{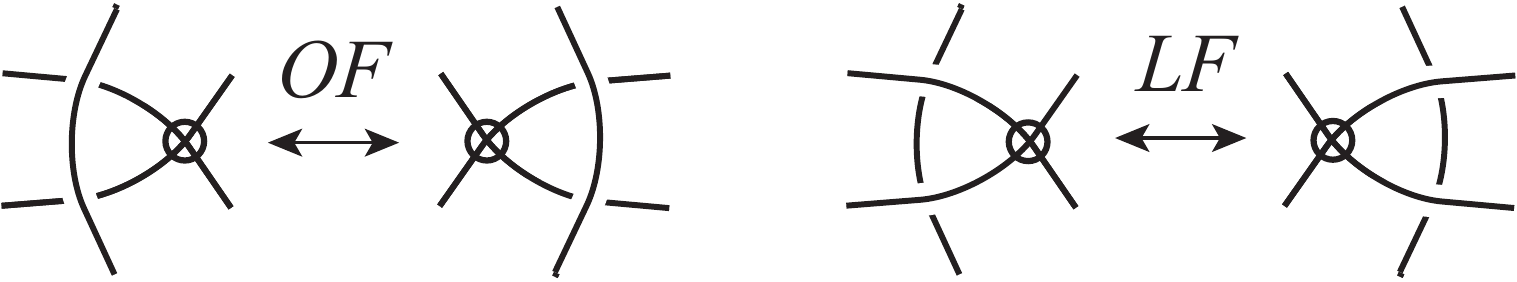}
\caption{Forbidden moves consisting of type $OF$ and type $LF$.}
\label{fig:forbidden}
\end{figure}

In the rest of the paperer,
if there is no confusion, a long virtual knot
(based Gauss diagram, resp.) is simply called a virtual knot (Gauss diagram, resp.).

\subsection{Invariants of ${\gpv}$-order and $F$-order}
\label{sub:FTI}

\begin{fact}[Goussarov, Polyak, and Viro \cite{GPV}]
Virtualization is an unknotting operation for virtual knots.
\end{fact}
\begin{fact}[Kanenobu \cite{kanenobu1}, Nelson \cite{nelson}]\label{fact2}
The pair $OF, LF$ of forbidden moves is an unknotting operation for virtual knots.  
\end{fact}

\begin{definition}[triangle]\label{dfn_triangle}
Let $D_K$ be a virtual knot diagram ($\subset \mathbb{R}^2$).  For $D_K$, if there exists a disk ($\subset \mathbb{R}^2$) which look like one of Fig.~\ref{fig:triangle disks}, the disk is called a {\it triangle}. For every triangle, the sign of the triangle is defined as shown in Fig.~\ref{fig:signs of triangle disks}.  Note that triangles consist of the four types.  
When we would like to specify the sign of a triangle, the triangle is called an $\epsilon$-triangle ($\epsilon = +, -$).  We call a triangle with the opposite sign to $\epsilon$ a $(-\epsilon)$-triangle.  

\begin{figure}[htbp]
\centering
\includegraphics[width=11cm]{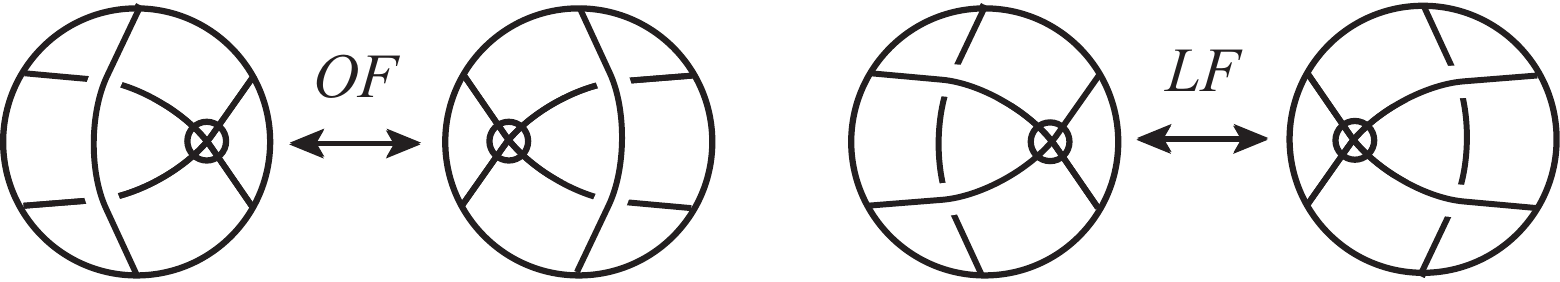}
\caption{Triangles (four types).}
\label{fig:triangle disks}
\end{figure}

\begin{figure}[htbp]
\centering
\includegraphics[width=12cm,clip]{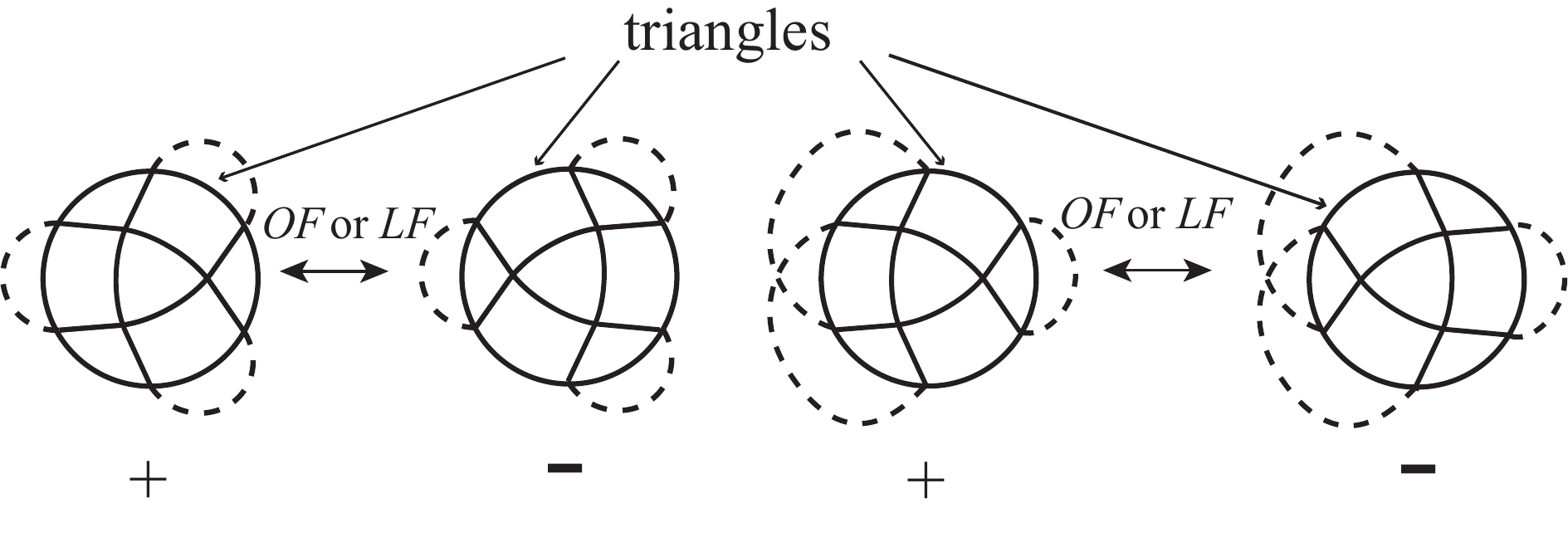}
\caption{Signs of triangles.  Dotted curves indicate the connections of virtual knot diagrams. }
\label{fig:signs of triangle disks}
\end{figure}
\end{definition}

\begin{definition}[Finite type invariants of $\gpv$-order]
Let $\mathcal{VK}$ be the set of the virtual knots. Let $G$ be an abelian group.  Let $v$ be a function $\mathcal{VK}$ $\to$ $G$.  We suppose that $v$ is an invariant of virtual knots.    
The function $v: {\mathcal{VK}} \rightarrow G$ is called a \emph{finite type invariant of $\gpv$-order} $\leq n$ if for any virtual knot diagram $D$ and for any $n+1$ real crossings $d_1, d_2, \ldots, d_{n+1}$,
\begin{align}
\label{align:gpveq}
\sum_{\delta} {(-1)^{|\delta |} v(D_{\delta })=0},
\end{align}
where $\delta=(\delta_1, \delta_2, \ldots , \delta_{n+1})$ runs over $(n+1)$-tuples of 0 or 1, $|\delta |$ is the number of $1$'s in $\delta$, and
$D_\delta$ is a diagram obtained from $D$ by applying a virtualization to $d_i$ with $\delta_i=1$.

If $v$ is a finite type invariant of $\gpv$-order $\leq n$ and is not a finite type invariant of $\gpv$-order $\leq {n-1}$, $v$ is called a \emph{finite type invariant of $\gpv$-order} $n$, and $v$ is denoted by $v^{\gpv}_n$.  
\end{definition}  

Let $\mathcal{VD}$ be the set of virtual knot diagrams and 
$\mathbb Z[{\mathcal{VD}}]$ the free $\mathbb Z$-module generated by the elements
of $\mathcal{VD}$. By the following formal relation (\ref{align:semiv}), we introduce a semi-virtual crossing, 
and a diagram having some semi-virtual crossings.

\begin{align}
\label{align:semiv}
\parbox{30pt}{\includegraphics[width=30pt]{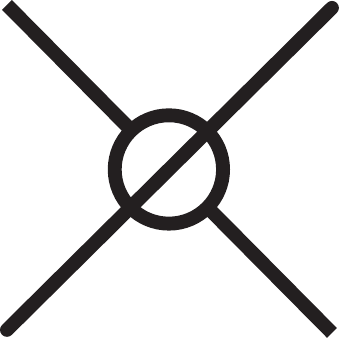}}~
:=~\parbox{30pt}{\includegraphics[width=30pt]{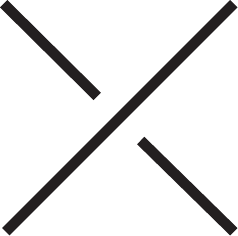}}~
-~\parbox{30pt}{\includegraphics[width=30pt]{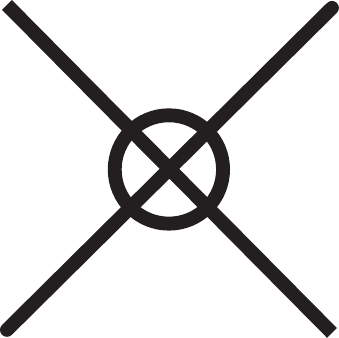}}
\end{align}

Any virtual knot invariant extends to formal linear combinations of
virtual knot diagrams by linearity. Under the identification (\ref{align:semiv}),
the alternating sum in the left hand side of the equality (\ref{align:gpveq})
becomes the value of $v$ on a virtual knot with $n+1$ semi-virtual crossings.
Thus, a virtual knot invariant has $\gpv$-order at most $n$ if its extension vanishes
on every a diagram having at least $n+1$ simi-virtual crossings.

\begin{definition}[Finite type invariants of $F$-order]\label{f-order}
Let $\mathcal{VK}$ be the set of the virtual knots.  Let $G$ be an abelian group  
and $v$ a function $\mathcal{VK}$ $\to$ $G$.  We suppose that $v$ is an invariant of virtual knots.  
The function $v :$ ${\mathcal{VK}} \rightarrow G$ is called a \emph{finite type invariant of $F$-order} $\leq n$ if for any virtual knot diagram $D$ and for any $n+1$ disjoint triangles $d_1, d_2, \ldots, d_{n+1}$,
\begin{align}
\label{align:Fneq}
\sum_{\delta} {(-1)^{|\delta |} v(D_{\delta })=0},
\end{align}
where $\delta=(\delta_1, \delta_2, \ldots , \delta_{n+1})$ runs over $(n+1)$-tuples of 0 or 1, $|\delta |$ is the number of $1$'s in $\delta$, and
$D_\delta$ is a diagram obtained from $D$ by applying a forbidden move to $d_i$ with $\delta_i=1$.

%In this paper, a finite-type invariant for forbidden moves of order $\leq n$ is simply called a \emph{finite-type invariant of $F$-order} $\leq n$.  
If $v$ is a finite type invariant of $F$-order $\leq n$ and is not a finite type invariant of $F$-order $\leq {n-1}$, $v$ is called a \emph{finite type invariant of $F$-order} $n$, and $v$ is denoted by $v^F_n$.    
\end{definition}  

Let $\mathcal{VD}$ be the set of virtual knot diagrams and 
$\mathbb Z[{\mathcal{VD}}]$ the free $\mathbb Z$-module generated by the elements
of $\mathcal{VD}$. By the following formal relation (\ref{align:semiOF}) and  (\ref{align:semiLF}), we introduce a semi-triple point, 
and a diagram having some semi-triple points.

\begin{align}
\parbox{40pt}{\includegraphics[width=40pt]{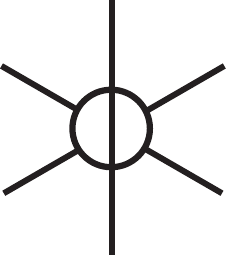}}~
&:=~\parbox{40pt}{\includegraphics[width=40pt]{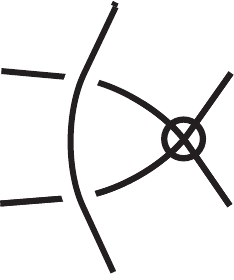}}~
-~\parbox{40pt}{\includegraphics[width=40pt]{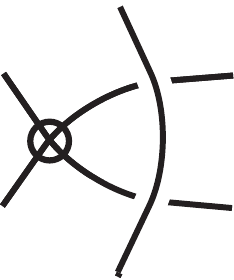}} ,\mbox{ and}\label{align:semiOF}\\
& \qquad {\textrm{positive}} \qquad {\textrm{negative}} \nonumber\\
\parbox{40pt}{\includegraphics[width=40pt]{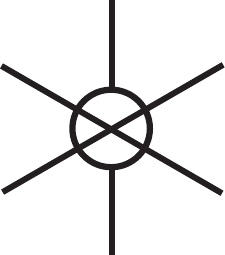}}~
&:=~\parbox{40pt}{\includegraphics[width=40pt]{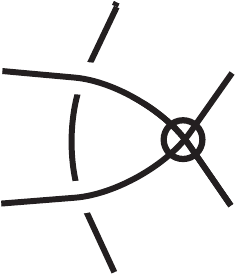}}~
-~\parbox{40pt}{\includegraphics[width=40pt]{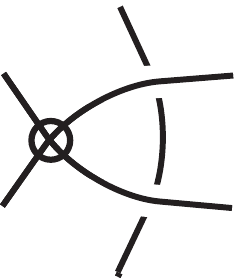}}. \label{align:semiLF}\\
& \qquad {\textrm{positive}} \qquad {\textrm{negative}} \nonumber
\end{align}

Any virtual knot invariant extends to formal linear combinations of
virtual knot diagrams by linearity. Under the identification (\ref{align:semiOF}) and  (\ref{align:semiLF}),
the alternating sum in the left hand side of the equality (\ref{align:Fneq}),
becomes the value of $v$ on a virtual knot with $n+1$ semi-triple points.
Thus, a virtual knot invariant has $F$-order at most $n$ if its extension vanishes
on every a diagram having at least $n+1$ simi-triple points.

\subsection{$\gpv _n$-similarity and $F_n$-similarity}
\label{sub:ntri}

\begin{definition}[$\gpv _n$-similar, $\gpv _n$-trivial]
Let $K$ and $L$ be virtual knots and
$\widetilde{K}$ ($\widetilde{L}$, resp.) a virtual knot diagram of $K$ ($L$, resp.).  
Suppose that $A_i$ ($1 \le i \le n$) is a non-empty set of some real crossings of $\widetilde{K}$ or $\widetilde{L}$.  
We say that $K$ and $L$ are \emph{$\gpv _n$-similar} 
if there exist $\{ A_i$ ($1 \leq i \leq n$)$\}$ satisfying two conditions:  
\begin{itemize}
\item $A_i \cap A_j = \emptyset$ $(i \neq j)$ for every pair $i$ and $j$, and 
\item by replacing real to virtual at every crossing of any nonempty subfamily of $\{A_i~|~1\leq i \leq n\}$, $\widetilde{L}$ is obtained from $\widetilde{K}$ or $\widetilde{K}$ is obtained from $\widetilde{L}$.  
\end{itemize}
In particular, if a virtual knot $K$ and the unknot are $\gpv _n$-similar, $K$ is \emph{$\gpv _n$-trivial}.   
\end{definition}
\begin{definition}[$F_n$-similar, $F_n$-trivial]
Let $K$ and $L$ be virtual knots, and $\widetilde{K}$ ($\widetilde{L}$, resp.) a virtual knot diagram of $K$ ($L$, resp.).  
Let $A_i$ ($1 \le i \le n$) be a non-empty set of disjoint triangles in $\widetilde{K}$.  
We say that $K$ and $L$ are \emph{$F_n$-similar} if there exist $\{ A_i$ ($1 \leq i \leq n$)$\}$ such that
\begin{itemize}
\item $A_i \cap A_j = \emptyset$ $(i \neq j)$, and 
\item $\widetilde{L}$ is obtained from $\widetilde{K}$ by forbidden moves at the triangles in any nonempty subfamily of $\{A_i~|~1\leq i \leq n\}$.  
\end{itemize}
In particular, if a virtual knot $K$ and the unknot are $F_n$-similar, $K$ is \emph{$F_n$-trivial}.
\end{definition}

In \cite{IS}, we have the following results.

\begin{fact}[Proof of Theorem 1, \cite{IS}]
Let $K$ be a given virtual knot. For any natural number $n$ and $\ell$, there exists $K_n^{\ell}$ such that $K_n^{\ell}$ is a $\gpv_n$-similar to $K$.
\label{fact:ex_GPV}
\end{fact}

\begin{fact}[\cite{IS}]
For any natural number $n$, there exists $K_n$ such that $K_n$ is a $F_n$-trivial virtual knot.
\label{fact:ex_Fn}
\end{fact}

As corollaries, we have Fact \ref{fact:GPV} and \ref{fact:Fn}.

\begin{fact}\label{fact:GPV} 
For any classical knot $K$, any positive integers $\ell$, $m$, and $n$ $(m \le n-1)$, and any finite type invariant $v^{\gpv}_m$ of $\gpv$-order $m$, there exist infinitely many of classical knots $K^{\ell}_n$ such that $v^{\gpv}_m(K^{\ell}_n)=v^{\gpv}_m(K)$.
\end{fact}
\begin{fact}\label{fact:Fn}
Let $O$ be the unknot.  For any positive integers $m$ and $n$ $(m \le n-1)$, and for any finite type invariant $v^F_m$ of $F$-order $m$, there exists a nontrivial virtual knot $K_n$ such that $v^F_m (K_n)=v^F_m (O)$.
\end{fact}

Here, by using Lemma \ref{lem:gpv-simi} (Lemma \ref{lem:Fnsimi}, resp.),
Fact \ref{fact:ex_GPV} (Fact \ref{fact:ex_Fn}, resp.) implies Fact \ref{fact:GPV} (Fact \ref{fact:Fn}, resp.)

\begin{Lemma}[Lemma~1 of \cite{IS}]
\label{lem:gpv-simi}
If $K$ is ${\gpv}_n$-similar to $K'$,  then
\begin{align*}
v^{\gpv}_m(K)&=v^{\gpv}_m(K') & (m<n).
\end{align*}
\end{Lemma}

\begin{Lemma}[Lemma~3 of \cite{IS}]
\label{lem:Fnsimi}
If $K$ is $F_n$-similar to $K'$, then
\begin{align*}
v^{F}_{m}(K)&=v^{F}_{m}(K') &(m<n).
\end{align*}
\end{Lemma}

%\begin{fact}[\cite{IS}]
%For any classical knot $K$, any positive integers $l$, $m$, and $n$ $(m \le n-1)$, and any finite type invariant $v^{\gpv}_m$ of GPV-order $m$, there exist infinitely many of classical knots $K^{\ell}_n$ such that $v^{\gpv}_m(K^{\ell}_n)=v^{\gpv}_m(K)$.
%\label{fact:ex_GPV}
%\end{fact}

%%%%%%%%%%%%%

\section{Main Results}
\label{sec:results}

For $F_n$-similarity, we generalize Fact \ref{fact:ex_Fn} to Theorem \ref{thm:Fn-simi}.  

\begin{theorem}\label{thm1}
Let $L$ be a given virtual link.  Let $n$ and $\ell$ be positive integers, there exists $L^\ell _n$ such that $L^\ell _n$ is $F_n$- similar to $L$.
\label{thm:Fn-simi}
\end{theorem}

\begin{theorem}\label{thm2}
Let $v^{\gpv}_i$ $($$v^{F}_i$, resp.$)$ be a finite type invariant of $\gpv$-order $i$ $($$F$-order $i$$)$.  Then,   
\[\{ v ~|~ v=v^{\gpv}_i (i \le 2n+1) \} \subset \{ v ~|~ v=v^{F}_i (i \le n) \}.\]  
\label{thm:GPVFn}
\end{theorem}
It is known that $\{v~|~ v=v_n^{\gpv}~~(n \in \mathbb{N}) \}$ is a complete invariant of virtual knots by \cite{GPV}.  Thus, by Theorem~\ref{thm:GPVFn}, $\{v~|~ v=v_n^F~~(n \in \mathbb{N}) \}$ is also a complete invariant of virtual knots.

\begin{proposition}
\label{prop:gpv&gpv}
$\{v~|~v=v_m^{\gpv}~(m \le 2(n+1)) \}$ is a strictly stronger long virtual knot invariant than $\{v~|~v=v_m^{\gpv}~(m \le 2n) \}$.  There exist two classical knots $K$ and $K'$ such that the former detects the difference of them and the latter does not.   \end{proposition}

\begin{theorem}\label{thm3}
For every order, non-trivial finite type invariants of $F$-order exist.

Further, denoting a finite type invariant of $F$-order $m$ by $v_m^F$, $\{v~|~v=v_m^F~(m \le n+1) \}$ is a strictly stronger long virtual knot invariant than $\{v~|~v=v_m^F~(m \le n) \}$.  There exist two classical knots $K$ and $K'$ such that the former detects the difference of them and the latter does not.   
\label{thm:nontriFn}
\end{theorem}

\begin{corollary}
$\{v~|~v=v_m^F~(m \le n+1) \}$ is a strictly stronger virtual long knot invariant than $\{ v ~|~ v=v^{\gpv}_m (m \le 2n+1) \}$.  There exist two classical knots $K$ and $K'$ such that the former detects the difference of them and the latter does not. 
\end{corollary}

\section{Proof of Theorem \ref{thm:Fn-simi}}
\label{sub:Fn-simi}

To begin with, we recall the definition of the virtual knot $\hat{b}(k)$,
which is introduced by \cite{IS} (for the detail of $\hat{b}(k)$, see \cite[Section 4.~2]{IS}).   Let 
\begin{align*}
A:&=\parbox{38pt}{\includegraphics[width=38pt]{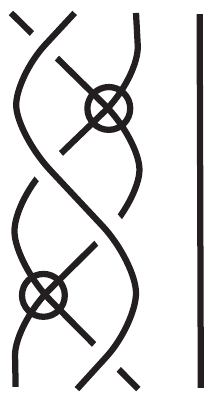}}~
&A^{-1}:=~\parbox{38pt}{\includegraphics[width=38pt]{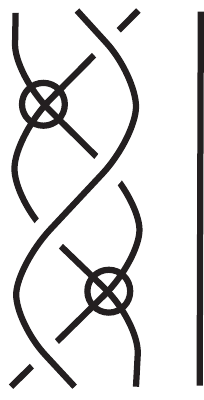}}~&
B:=~\parbox{38pt}{\includegraphics[width=38pt]{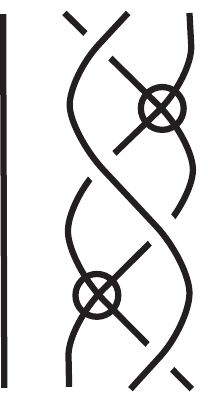}},~{\textrm{and}}
&B^{-1}:=~\parbox{38pt}{\includegraphics[width=38pt]{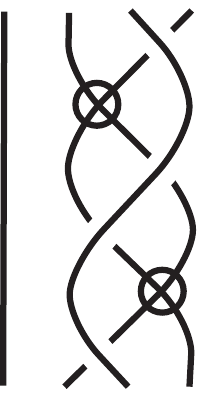}}, 
\end{align*}
and let $b(1)=A,$ $b(2)=[B,b(1)]$, $b(4u-1)=[B,b(4u-2)]$, $b(4u)=[A,b(4u-1)]$, $b(4u+1)=[A,b(4u)]$, $b(4u+2)=[B,b(4u+1)]$,  $(u \geq 1).$  The closure of $b(k)$ is defined by Fig.~\ref{closure}.

\begin{figure}[h!]
\centering
\includegraphics[width=4cm]{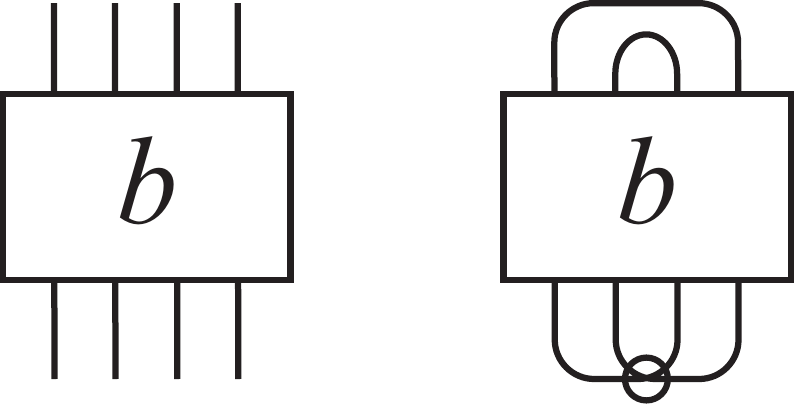}
\vspace{-0.3cm}
\caption{A braid $b$ and its closure $\hat{b}$.}\label{closure}
\end{figure}

\begin{figure}[htbp]
\centering
\includegraphics[width=10cm]{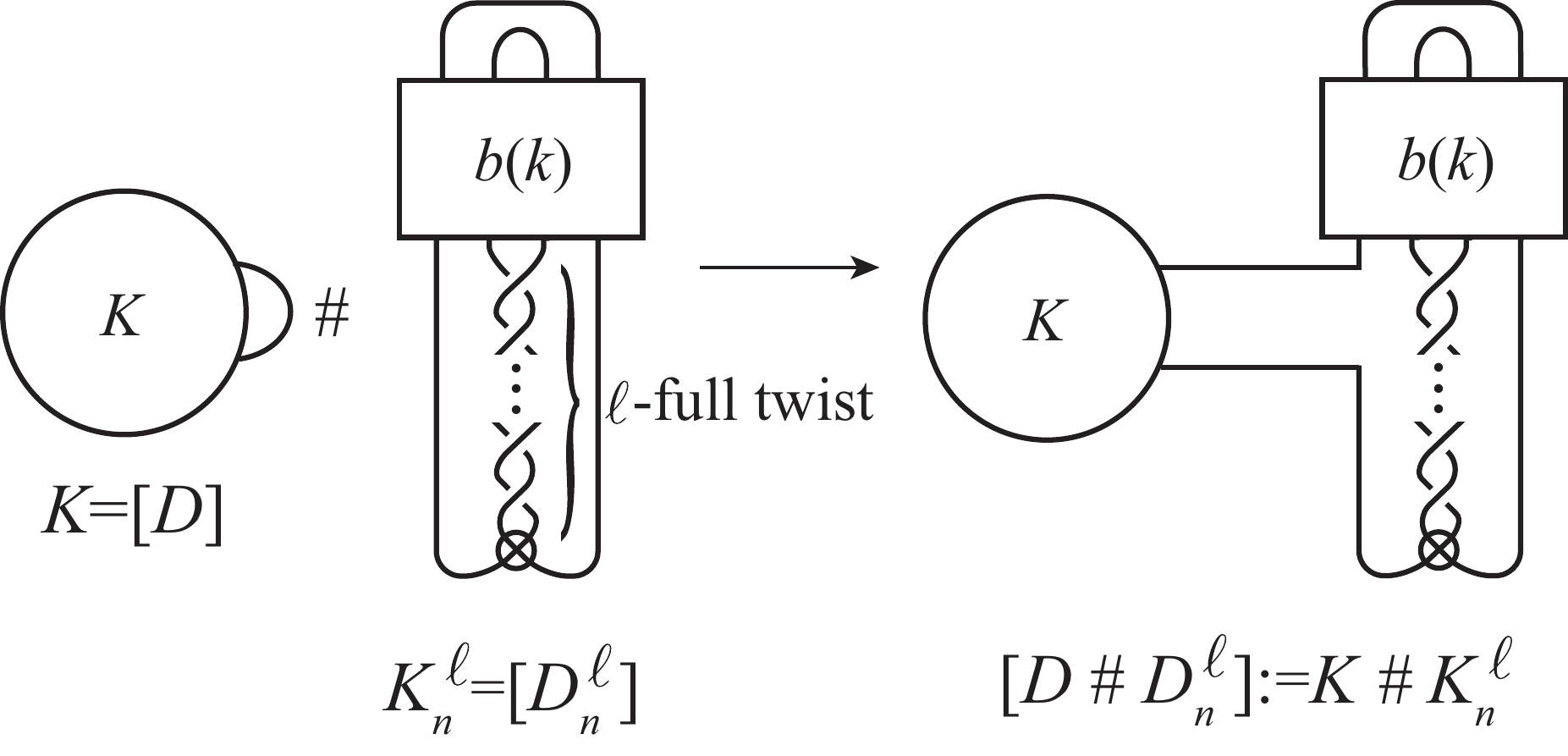}
\caption{We define $K \# K_n^{\ell}$ consisting of virtual knots $K$ and $K_n^{\ell}$ by the local
replacement for their virtual knot diagrams $D$ and $D_n^{\ell}$.  Square
brackets indicate equivalence classes of virtual knot diagrams.}
\label{connectedsum}
\end{figure}

%\begin{corollary}
%Let $m$, $n$ be a given pair of positive integers satisfying $m \leq n-1$ and fixed.  For any virtual knot $K$, any positive integer $\ell$, there exist infinitely many virtual knots $K^{\ell}_n$ such that $v^{F}_m(K \# K^{\ell}_n)=v^{F}_m(K)$.
%\label{thm:vFn-simi}
%\end{corollary}

Next, we start the proof of Theorem \ref{thm:Fn-simi}.  

\begin{proof}
(Step~1) Let $J(L)$ be a Jones polynomial of a given virtual link $L$, is defined by \cite[Formula~(1)]{manturov} (in \cite{manturov}, it is denoted by $X(L)$).  
Let $\gamma_0$ $=$ the maximal degree of the Jones polynomial $J (\hat{b}(n))$ ($n \in \mathbb{N}$).    
Let $S_{\gamma _0}$ be a Kauffman state with the maximal degree $\gamma _0$. Let $K^{0}_n$ $=$ $\hat{b}(n)$ and let $K^{\ell}_n$ be  a virtual knot obtained by applying $\ell$-full twists
to $\hat{b}(n)$.  By definition, using $S_{\gamma_0}$, there exists a Kauffman state implying the maximal degree of $J(K^{\ell}_n)$.
Then, the degree is represented by $\gamma_0 + 2 \ell$.
Thus, $K_n^{\ell} \neq K_n^{\ell '}$ $(\ell < \ell ')$.

\begin{figure}[htbp]
\centering
\includegraphics[width=6cm]{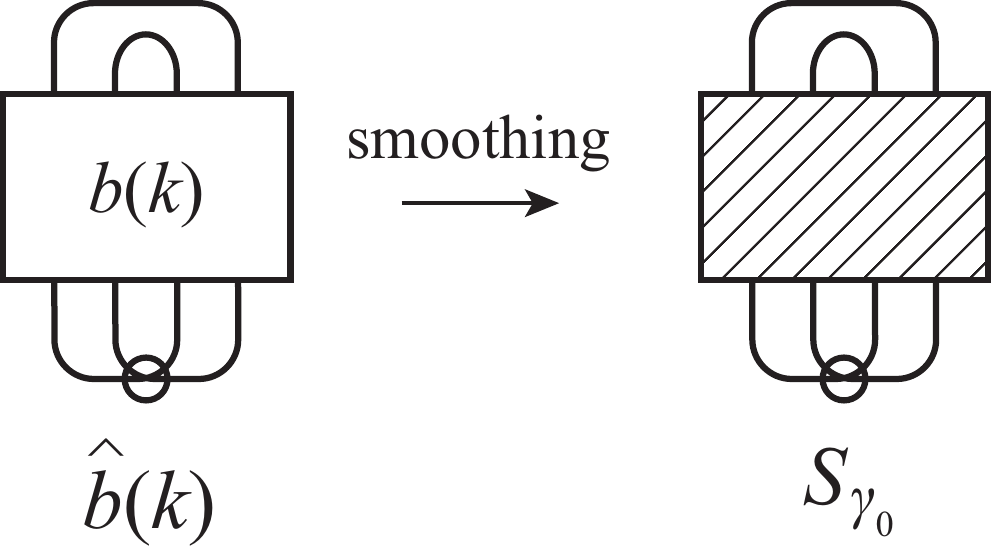}
\end{figure}

\begin{figure}[htbp]
\centering
\includegraphics[width=10cm]{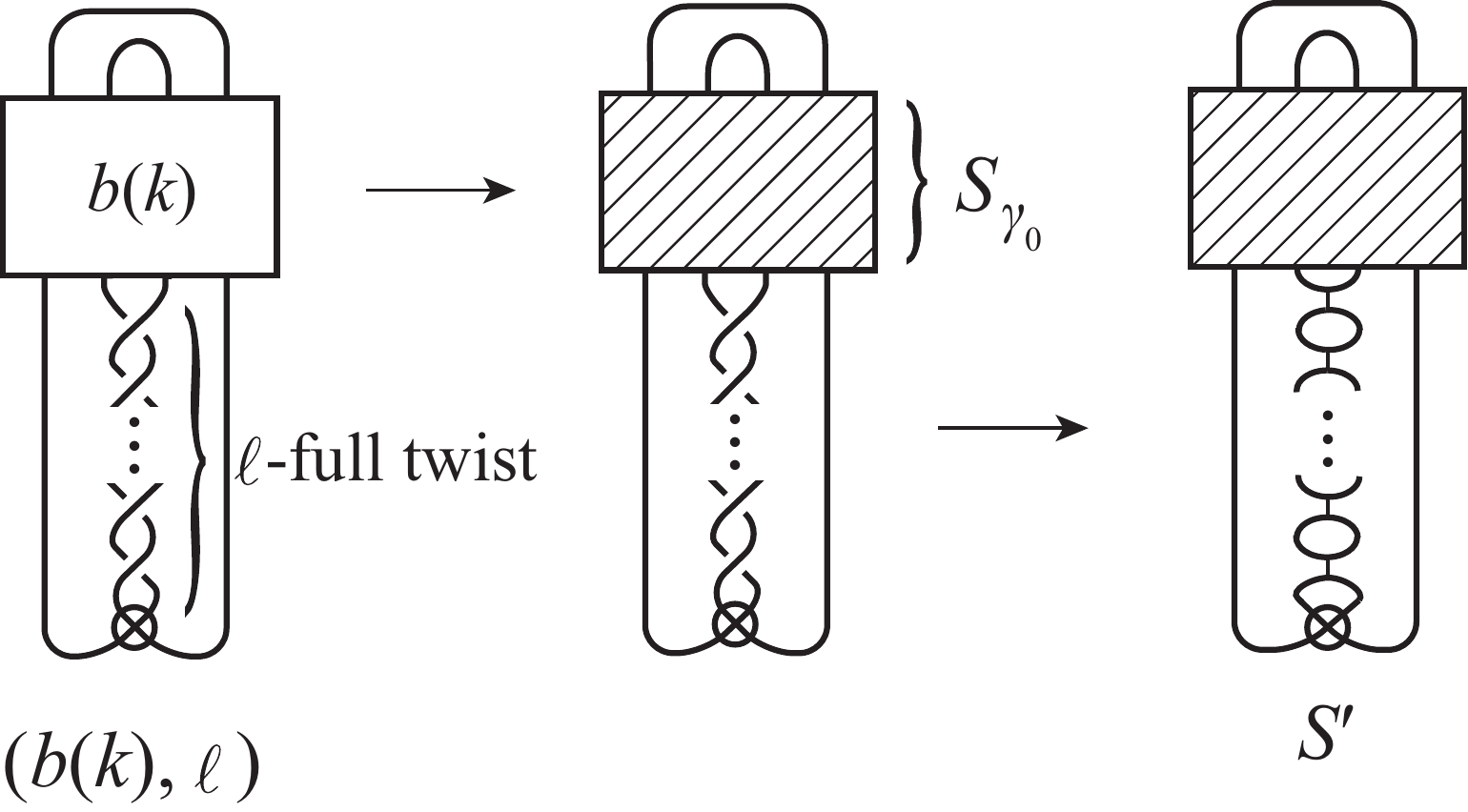}
\end{figure}

\noindent (Step~2) Let $L$ be a given virtual knot and let $\gamma_1$ be the maximal degree of $J(L)$.   
We firstly consider the case that $L$ is a virtual knot, and set $L=K$.    
Let $D$ be a virtual knot diagram and $D^{\ell}_n$ a virtual knot  diagram given by $\hat{b}(n)$ and $l$-full twists, as shown in the left half of Fig.~\ref{connectedsum}.     
Then, let $K \# K_n^{\ell}$ be the virtual knot having a virtual knot diagram $D \# D_n^{\ell}$,
as shown in the right half of Fig \ref{connectedsum}.    
Then, the maximal degree of $J (K \# K_n^{\ell})$ is $\gamma_1 + \gamma_0 + 2\ell-1$.  Here, if $\gamma_1 + \gamma_0 -1 <0$, we replace $\gamma _0$ with $\gamma _0' +2\ell _0$
by applying $\ell _0$-full twists, where $\ell_0$ is a sufficient large.  Then, it implies $\gamma_1 + \gamma_0' -1 \geq 0$  
Then, by comparing the maximal degree of $K_n^{\ell}$ and that of $K_n^{\ell'}$ $(\ell < \ell')$, we have $K \# K_n^{\ell} \neq K \# K_n^{\ell'}$.

We secondly consider the case that $L$ is a virtual link, but it is easy to extend the above case to the link case.  
\end{proof}

\section{High-order of $v_n^{\gpv}$ and $v_n^F$: proofs of Theorems~\ref{thm2} and \ref{thm3}}
\label{sec:GPVFn}
We prepare Lemma \ref{lem:F2GPV}.
Every forbidden move consists of exactly two virtualizations and generalized Reidemeister moves.

\begin{Lemma}
\label{lem:F2GPV}
Let $A_1$ and $A_2$ be the set of real crossings, as shown in Fig.~\ref{fig:GPVFn}.  Then, a single forbidden move is realized by virtualization at the crossings in $A_1$ $(A_2$ or $A_1 \cup A_2$, resp.$)$ up to generalized Reidemeister moves.  
\begin{figure}[htbp]
\centering
\includegraphics[width=12cm]{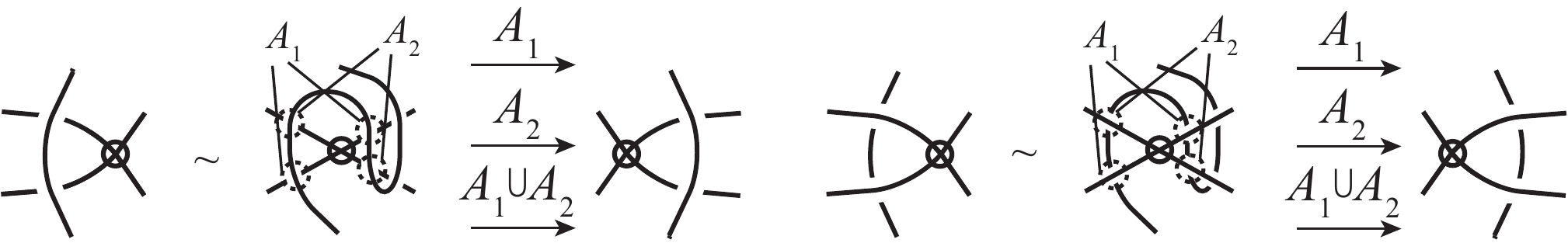}
\caption{forbidden move and virtualization}
\label{fig:GPVFn}
\end{figure}
\end{Lemma}

\begin{proof}
First, consider $OF$.
If we apply virtualizations to crossings in $A_1$, $A_2$ and $A_1 \cup A_2$, 
we have one from the other of $OF$, as shown in (i), (ii) and (iii) of Fig.~\ref{fig:KandK},
respectively.

\begin{figure}[htbp]
\centering
\includegraphics[width=10cm]{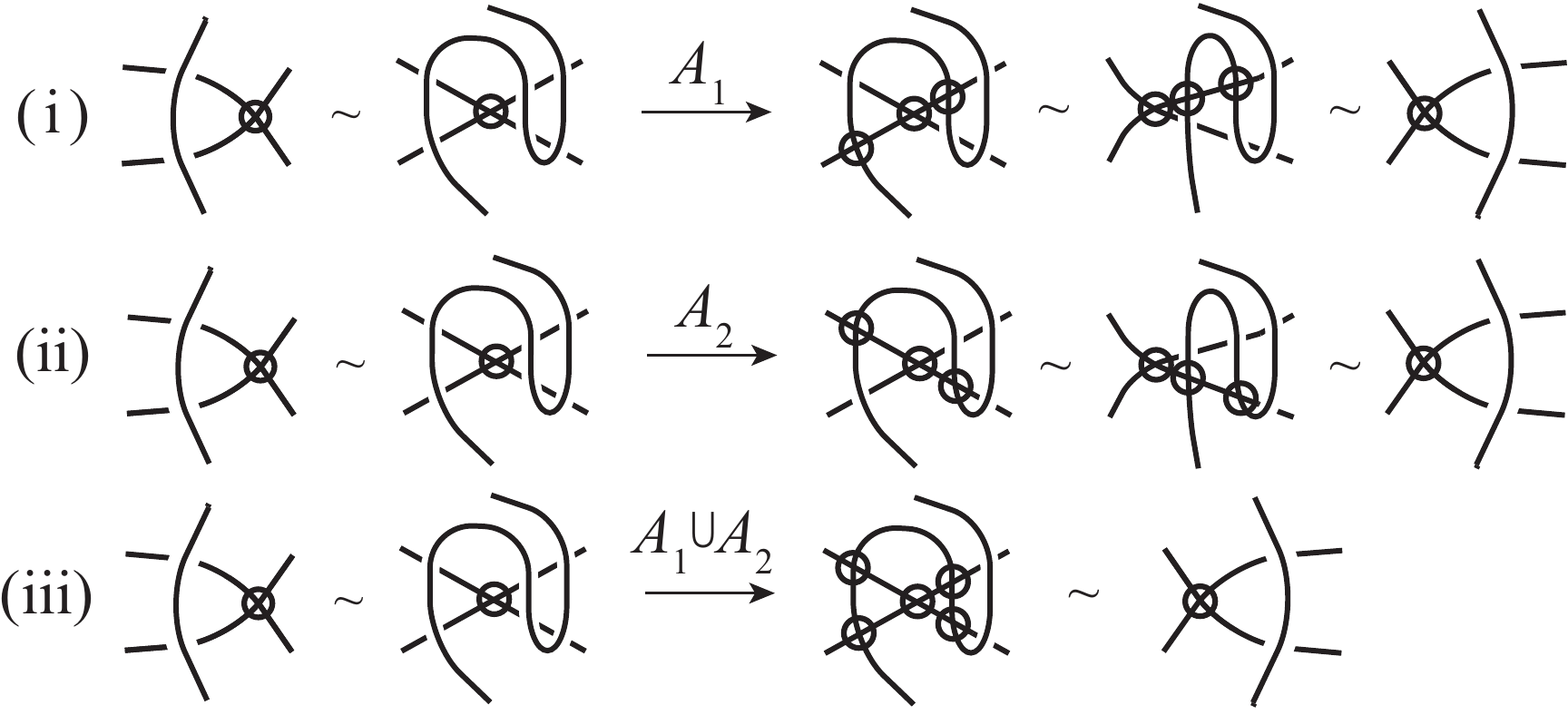}
\caption{The deformation of triangles of $OF$.}
\label{fig:KandK}
\end{figure}

Second, if we apply virtualizations to crossings in $A_1$, $A_2$ and $A_1 \cup A_2$ of $LF$, we also have one from the other for each case.

\end{proof}

\noindent ({\bf Proof of Theorem~\ref{thm:GPVFn}}.)
By Lemma \ref{lem:F2GPV} and relations (\ref{aling:triple1}) and (\ref{aling:triple2}),

\vspace{-0.2cm}
\begin{align}
v\left(~\parbox{30pt}{\includegraphics[width=30pt]{triplepoint1.pdf}}~\right)
&=v\left(~\parbox{30pt}{\includegraphics[width=30pt]{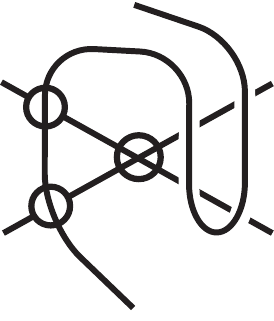}}~\right)
+v\left(~\parbox{30pt}{\includegraphics[width=30pt]{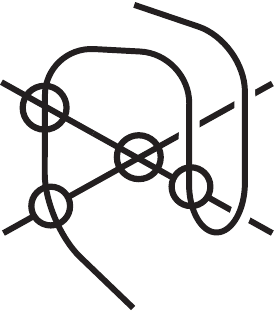}}~\right)  \label{aling:triple1}  \\ 
&+v\left(~\parbox{30pt}{\includegraphics[width=30pt]{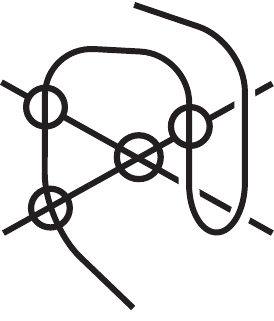}}~\right)
+v\left(~\parbox{30pt}{\includegraphics[width=30pt]{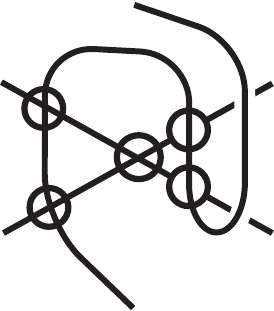}}~\right) \mbox{, and}\notag
\end{align}

\begin{align}
v\left(~\parbox{30pt}{\includegraphics[width=30pt]{triplepoint2.pdf}}~\right)
&=v\left(~\parbox{30pt}{\includegraphics[width=30pt]{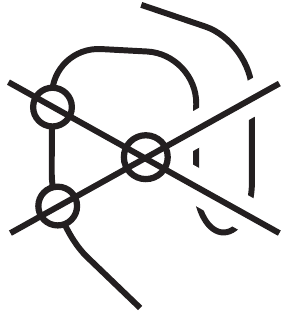}}~\right)
+v\left(~\parbox{30pt}{\includegraphics[width=30pt]{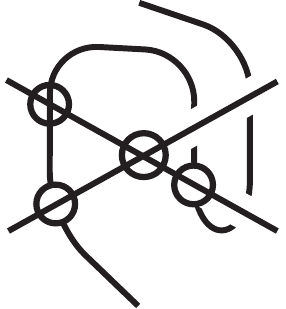}}~\right) \label{aling:triple2}\\ 
&+v\left(~\parbox{30pt}{\includegraphics[width=30pt]{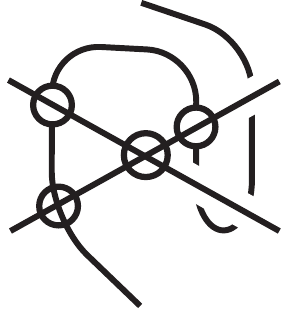}}~\right)
+v\left(~\parbox{30pt}{\includegraphics[width=30pt]{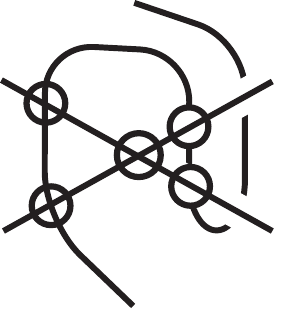}}~\right). \notag 
\end{align}

Let $D_n^T$ be a virtual knot diagram with $n$ semi-triple points
and let $D_n^S$ a virtual knot diagram with $n$ semi-virtual crossings.

Then, by the above equations (\ref{aling:triple1}) and (\ref{aling:triple2}), %${}^\exists \lambda _i$ : coeff. $(i \leq 2n+1)$

\[
v^{\gpv}_{2n+1} (D^{T}_{n+1})\stackrel{(\ref{aling:triple1}), (\ref{aling:triple2})}{=}\sum_{D^{S}_{2(n+1)}} v^{\gpv}_{2n+1} (D^{S}_{2(n+1)})  
\stackrel{\mbox{by def.}}{=} 0.
\]

Similarly, 
\[
v^{\gpv}_{2n} (D^{T}_{n+1}) \stackrel{(\ref{aling:triple1}), (\ref{aling:triple2})}{=} \sum_{D^{S}_{2(n+1)}} v^{\gpv}_{2n} (D^{S}_{2(n+1)})  \stackrel{\mbox{by def.}}{=} 0.
\]
$\hfill\Box$

By Lemma \ref{lem:F2GPV}, we have Corollary~\ref{cor2}.   
\begin{corollary}\label{cor2}
If $K$ is $F_{n+1}$-similar to $K'$, then, $K$ is  ${\gpv}_{2(n+1)}$-similar to $K'$.    
\end{corollary}

Let $\mathcal P$ : Polayak algebra, i.e., ${\gpv}_n$-inv. $\in {\mathcal P}_{n}^*$.

\begin{fact}[\cite{GPV}]
Let $D$ be any diagram of a virtual knot $K$. The formula $K \mapsto I(D) \in {\mathcal P}$ defines a complete invariant of virtual knots.
\label{thm:gpv}
\end{fact}

By using Fact~\ref{thm:gpv}, we have Corollary \ref{coro:compFn}.

\begin{corollary}
The set of finite type invariants of $F$-order gives a complete invariant of virtual knots. 
\label{coro:compFn}
\end{corollary}

\begin{fact}[\cite{CKR}]
For $n \geqq 1$, the coefficient $c_{2n}$ of $z^{2n}$ in the Conway polynomial of
a knot $K$ with the Gauss diagram $G$ is equal to
$$
c_{2n} = \langle {\mathfrak C}_{2n}, G \rangle.
$$
\label{thm:conway}
\vspace{-0.5cm}
\end{fact}

The definition of the Conway polynomial is extended to that of long virtual knots 
\cite{CKR,P}. 
The invariance of the Conway polynomial of long virtual knots is obtained by
Polyak \cite{P}.
This result of Polyak with the skein relation given by \cite{CKR}
implies that the Gauss diagram formula of \cite{CKR} are extended to
those of long virtual knots. By the definition of ${\gpv}$-invariants, we have Proposition~\ref{prop:gpvVa}.

\begin{notation}
For a Gauss diagram $H$, let $|H|$ be the number of chords.  
\label{nota:gauss}
\end{notation}

%By definition,
%$\langle A, J(D) \rangle = \sum_{z \in {\mathrm Sub}(D)}z$.

\begin{proposition}
Let $G$ be a Gauss diagram of a virtual knot.
Let $\langle \sum_{i=1}^{t} \lambda _i A_i, G \rangle$ be a Gauss diagram formula of a
virtual knot and $n_0=\max \{ |A_i| | 1 \leq i \leq t \}$.
Then, $\langle \sum_{i=1}^{t} \lambda _i A_i, G \rangle$ is  a finite type invariant of ${\gpv}$-order $\leq n_0$.
\label{prop:gpvVa}
\end{proposition}

\begin{proof}
Without loss of generality, we may $|A_1|=n_0$ and then
we may denote $A_1$ by $A$ in the following.
Suppose that a diagram $D_m^{(p)}$ has $p$ semi-virtual crossings and $m-p$ real crossings.
Let $D_0$ be a virtual knot diagram obtained by replacing each semi-virtual crossing 
with the real crossing.
Let $G_0$ be a Gauss diagram corresponding to $D_0$.
Suppose that $n_0 < p$.

Note that $p$ chords of $G_0$ are obtained by replacing  semi-virtual crossings with real crossings.
Then, if $G'$ is given by erasing some of $p$ chords of $G$, 
we denote the relation by $G' \prec G$. By definition, if $G' \prec G$, the $G' \subset G$.  

By definition, we have a linear sum of virtual knot diagrams
obtained by $D^{(p)}_m$ by 
$
\parbox{30pt}{\includegraphics[width=30pt]{semi.pdf}}~
=~\parbox{30pt}{\includegraphics[width=30pt]{crossing1.pdf}}~
-~\parbox{30pt}{\includegraphics[width=30pt]{virtual2.pdf}}.
$
The linear sum of virtual knot diagrams gives the linear sum $\sum_{G' \prec G_0} (-1)^{|G_0|-|G'|}G'$. In order to obtain the claim of Proposition \ref{prop:gpvVa}, it is sufficient to show that $\langle A, \sum_{G' \prec G_0} (-1)^{|G_0|-|G'|}G'\rangle =0$. 

\begin{align}
\sum_{G' \prec G_0} (-1)^{|G_0|-|G'|}\langle A, G'\rangle
&= \sum_{G' \prec G_0} (-1)^{|G_0|-|G'|} \sum_{z \subset G'}( A, z )   \nonumber \\   
&= \sum_{z \subset G_0} \sum_{G'(z) \prec G_0} (-1)^{|G_0|-|G'(z)|} ( A, G'(z) )
\label{gaussF}
\end{align}

Let $q=|G_0|-|z|$. Let $c_1,c_2 \ldots , c_q$ be real crossings which are 
erased to obtain $z$ where $z \prec G_0$.
Then, suppose that the crossings witch are added to $z$ to obtain $G'(z)$
are $c_1', c_2', \ldots , c'_k \in \{ c_i ~|~ 1\leq i \leq q \}$.
We define $\sigma_i (1\leq i \leq q)$ by $\sigma_i =1$ if $c_i \in \{ c_i' ~|~ 1\leq i \leq k \}$
and $\sigma_i =0$ otherwise. Let $z(\sigma_1, \ldots , \sigma_q)=(-1)^{|G_0|-|z|}(-1)^{-|\sigma_1 + \ldots + \sigma_q|}$. It is easy to see that the following claim.

\noindent
{\bf Claim.}
\[z(\sigma_1, \ldots , \sigma_q)=(-1)^{|G_0|-|G'(z)|}.\]
Further, Let $k=|\sigma_1+ \cdots + \sigma_q|$.
Then, \[\sum_{G'(z) \prec G_0}(-1)^{|G_0|-|G'(z)|}= \sum_{k=0}^q (-1)^{q-k} 
\left(
\begin{matrix}
q\\
k
\end{matrix}
\right). \]

\begin{align*}
{\textrm{The right of~}}(\ref{gaussF})
=&\sum_{z \prec G_0} \left\{ \sum_{0< p-q \leq n_0} z(\sigma_1, \ldots , \sigma_q)(A, z)  
\right\}
\\
&+z(\sigma_1, \ldots , \sigma_p)(A, z) + \sum_{p-q > n_0} z(\sigma_1, \ldots , \sigma_q)(A, z)
\end{align*}

Here, by the assumption, we have $n_0<p$, which implies $|A|\neq |z|$.
Then $$z(\sigma_1, \ldots , \sigma_p)(A, z)=0.$$
Since $|z|>n_0=|A|$, 
$$
\sum_{p-q > n_0} z(\sigma_1, \ldots , \sigma_q)(A, z)=0.
$$
Further, by the above claim again,  
\begin{align*}
\sum_{G' \subset G_0} (-1)^{|G_0|-|G'|}\langle A, G'\rangle
&=\sum_{z \prec G_0} \left\{ \sum_{0< p-q \leq n} z(\sigma_1, \ldots , \sigma_q)(A, z)  
\right\}\\
&=\sum_{z \prec G_0} (A, z) \left\{ \sum_{k=0}^q (-1)^{q-k} 
\left(
\begin{matrix}
q\\
k
\end{matrix}
\right)
\right\} \\
&=\sum_{z \prec G_0} (A, z) \left\{ 1+(-1) \right\}^q \\
&=0.
\end{align*}
\end{proof}

\noindent ({\bf Proof of Theorem~\ref{thm3}}.) 
By using Proposition~\ref{prop:gpvVa}, Fact~\ref{thm:conway}, and the following argument, we have Proposition \ref{prop:gpv&gpv} which implies Theorem~\ref{thm3}.

By Theorem~\ref{thm2}, $c_{2n} \in \{v~|~v=v_{2m}^{\gpv} \quad (m \le n) \}$ $\subset \{v~|~v=v_{m}^{F} \quad (m \le n) \}$.  Thus, letting $\lambda^{(2m)}_i$ ($0 \le i \le m$) be coefficients, we have 
\[
c_{2m} = \lambda^{(2m)}_{m} v_m^F + \sum_{i \le m-1} \lambda^{(2m)}_i v_i^F.  
\]
In particular,
\[
c_{2n+2} = \lambda^{(2n+2)}_{n+1} v^F_{n+1} + \sum_{i \le n} \lambda^{(2n+2)}_i v_i^F. 
\]
Here, we suppose that $\lambda^{(2n)}_{n} \neq 0$, and will prove $\lambda^{(2n+2)}_{n+1} \neq 0$ for the induction.  
By using the assumption of the induction, there exist coefficients $\mu_i$ ($0 \le i \le m$) such that 
\begin{align}
c_{2n+2} &= \lambda^{(2n+2)}_{n+1} v^F_{n+1} + \sum_{i \le n} \mu_i c_{2i}.  
\label{aling:c2n+2}
\end{align}
If $\lambda^{(2n+2)}_{n+1} = 0$, then, for $K_{2n+2}$,
Equation (\ref{aling:c2n+2}) implies 
\[ -2 =  LHS = RHS = 0, \]
which implies a contradiction.   
Therefore, $c_{2n} \in \set{v| v=v_{2n}^{\gpv}} $.

$\hfill\Box$

\end{document}